\documentclass[arxiv,reqno,twoside,a4paper,12pt]{amsart}

\usepackage{tikz}
\usepackage{amsmath, verbatim}
\usepackage{amssymb,amsfonts,mathrsfs}
\usepackage[colorlinks=true,linkcolor=blue,citecolor=blue]{hyperref}
\usepackage[driver=pdftex,margin=3cm,top=3cm, bottom=3cm,centering]{geometry}
\usepackage{longtable}
\usepackage{tabularx}
\usepackage{multirow}
\usepackage{caption}

\usepackage[scaled]{helvet} % ss
\usepackage{courier} % tt
\usepackage[mathbf]{euler}

\theoremstyle{plain}
\newtheorem{theorem}{Theorem}[section]
\newtheorem{thm}{Theorem}[section]
\newtheorem{prop}[thm]{Proposition}
\newtheorem{remark}[thm]{Remark}
\newtheorem{rem}[thm]{Remark}
\newtheorem{cor}[thm]{Corollary}
\newtheorem{defn}[thm]{Definition}
\newtheorem{lem}[thm]{Lemma}

\theoremstyle{definition}

\newtheorem{assump}[theorem]{Assumption}

\theoremstyle{plain}

\newcommand{\Z}{\mathbb{Z}}
\newcommand{\R}{\mathbb{R}}

\newcommand{\N}{\mathbb{N}}

\newcommand{\CP}{\mathbb{CP}}
\newcommand{\HP}{\mathbb{HP}}

\newcommand{\dimn}{\mathrm{dim}}

\newcommand{\supp}{\mathrm{supp}}

\newcommand{\ric}{\mathrm{Ric}}

\newcommand{\Spec}{\mathrm{Spec}}

\DeclareMathOperator{\cH}{\mathscr{H}}
\DeclareMathOperator{\A}{\alpha}
\DeclareMathOperator{\w}{\omega}
\DeclareMathOperator{\V}{\mathcal{V}}
\DeclareMathOperator{\cC}{\mathscr{C}}
\DeclareMathOperator{\dom}{\mathscr{D}}

\newcommand{\wt}{\widetilde}
\newcommand{\wh}{\widehat}
\newcommand{\ov}{\overline}
\newcommand{\D}{\Delta}
\renewcommand{\d}{\delta}
\newcommand{\la}{\langle}
\newcommand{\ra}{\rangle}
\newcommand{\ve}{\varepsilon}
\newcommand{\g}{\gamma}
\renewcommand{\a}{\alpha}

\newcommand{\G}{\Gamma}

\DeclareMathOperator{\Rm}{Rm}
\newcommand{\p}{\partial}
\newcommand{\n}{\nabla}
\renewcommand{\supp}{\operatorname{supp}}
\newcommand{\s}{\sigma}

\DeclareMathOperator{\ho}{\mathcal{C}^{\A}_{\textup{ie}}}

\DeclareMathOperator{\hok}{\mathcal{C}^{k, \A}_{\textup{ie}}}

\numberwithin{equation}{section}

\definecolor{qqwuqq}{rgb}{0,0,0}

\begin{document}
	
	\date{\today}
	
	\title[Bounded Ricci Curvature and Positive Scalar Curvature under Ricci Flow]
	{Bounded Ricci Curvature and Positive Scalar Curvature under Singular Ricci de Turck Flow}

	\author{Klaus Kr\"oncke}
	\address{University Hamburg, Germany}
	\email{klaus.kroencke@uni-hamburg.de}
	
\author{Tobias Marxen}
	\address{University Oldenburg, Germany}
	\email{tobias.marxen@uni-oldenburg.de}

	\author{Boris Vertman}
	\address{University Oldenburg, Germany}
	\email{boris.vertman@uni-oldenburg.de}

	\thanks{Partial support by DFG Priority Programme "Geometry at Infinity"}
	
	
	\begin{abstract}
	In this paper we consider a Ricci de Turck flow of spaces with isolated conical singularities, 
	which preserves the conical structure along the flow. We establish that a given initial regularity of Ricci 
	curvature is preserved along the flow. Moreover under additional assumptions, positivity of scalar curvature 
	is preserved under such a flow, mirroring the standard property of Ricci flow on compact manifolds.  
	The analytic difficulty is the a priori low regularity of scalar curvature at the conical tip along the flow, so that
	the maximum principle does not apply. We view this work as a first step toward studying 
	positivity of the curvature operator along the singular Ricci flow.
	\end{abstract}
	
\maketitle
\tableofcontents

%%%%%%%%%%%%%%%%%%%%%%%%%%
\section{Introduction and statement of main results} \medskip
%%%%%%%%%%%%%%%%%%%%%%%%%%

Consider a compact smooth Riemannian manifold $(M,g)$ without boundary. 
Its Ricci flow is a smooth family of metrics $(g(t))_{t\geq 0}$ such that 
\begin{align}\label{RF}
\partial_t g(t) = -2 \textup{Ric} (g(t)), \quad g(0) = g,
\end{align}
where $\textup{Ric} (g(t))$ denotes the Ricci curvature tensor of $g(t)$.
Due to diffemorphism invariance of the Ricci tensor, this evolution equation 
fails to be strongly parabolic. One overcomes this problem by adding an additional 
term to the equation which brakes the diffeomorphism invariance. This leads to
the equivalent and analytically more convenient Ricci de Turck flow
\begin{equation}\label{RTF}
\partial_t g(t) = -2 \textup{Ric} (g(t)) + \mathcal{L}_{W(t)} g(t), \quad g(0) = g,
\end{equation}
where $W(t)$ is the de Turck vector field defined in terms of the Christoffel symbols for 
the metrics $g(t)$ and a background metric $h$ 
\begin{equation}
W(t)^k = g(t)^{ij} \left(\Gamma^k_{ij}(g(t)) - \Gamma^k_{ij}(h)\right). 
\end{equation}
In the applications, the background metric $h$ is usually taken as the initial metric $g$, 
or as its small Ricci flat perturbation. The de Turck vector field defines 
a one-parameter family of diffeomorphisms $(\phi(t))_{t\geq 0}$
and if $(g(t))_{t\geq 0}$ is a solution to the Ricci de Turck flow \eqref{RTF}, then the pullback
$(\phi(t)^*g(t))_{t\geq 0}$ is a solution to the Ricci flow \eqref{RF}.
\medskip

On singular spaces, the de Turck vector field may point towards the singular stratum and 
thus lengths of the corresponding integral curves may not be bounded from below away from 
zero. Hence the one-parameter family of diffeomorphisms $(\phi(t))_{t\geq 0}$ may 
not exist for positive times. Therefore, in the singular setting \eqref{RF} and \eqref{RTF} are generally
not equivalent and we study the latter flow. \medskip

Our work establishes an interesting property of the flow, namely that a given 
initial regularity of the Ricci curvature is preserved along the flow. 
\begin{thm}\label{main2}
The singular Ricci de Turck flow preserves the initial regularity of the Ricci curvature. 
% es ist "initial regularity" statt "initial singular structure". Bitte so lassen, wie jetzt.
In particular, if the initial metric has bounded Ricci curvature, the flow remains of bounded 
Ricci curvature as well.
\end{thm}

We are not able to deduce an analogous result for the scalar curvature.
The reason is that the norm of the Ricci tensor appears as the reaction term in the evolution 
equation of the scalar curvature. Thus, unbounded Ricci curvature at the singularity pushes 
the scalar curvature to infinity after infinitesimal time. In contrast, the evolution equation on the 
Ricci curvature is tensorial which allows more flexibility.
However, we are able to prove a different property of the scalar curvature along the Ricci flow 
which is well known in the smooth compact case.
\begin{thm}\label{main1}
An admissible Riemannian manifold with isolated conical singularities and positive scalar curvature 
admits a singular Ricci de Turck flow preserving the singular structure and unter the additional assumption of strong tangential stability, 
positivity of scalar curvature along the flow.
\end{thm}

The present work is a continuation of a research program on singular 
Ricci flow, that preserves the initial singular structure, that has seen 
several recent advances. In the setting of surfaces with conical singularities, singular 
Ricci flow has been studied by Mazzeo, Rubinstein and Sesum in \cite{MRS} and Yin \cite{Yin:RFO}. 
The Yamabe flow, which coincides with the Ricci flow in the two-dimensional setting, has been studied in general dimension on spaces with edge singularities
by Bahuaud and the third named author in \cite{BV} and \cite{BahVer}. In the setting of K\"ahler manifolds, K\"ahler-Ricci flow on spaces with edge singularities 
appears in the recent results on the Calabi-Yau conjecture on Fano manifolds, cf. Donaldson \cite{Donaldson}, Tian \cite{Tian}, see also Jeffres, Mazzeo and 
Rubinstein \cite{JMR}. K\"ahler-Ricci flow in case of isolated conical singularities has been addressed by Chen and Wang \cite{Wang1}, 
Wang \cite{Wang2}, as well as Liu and Zhang \cite{LZ}. \medskip

Ricci flow on singular spaces of general dimension, without the K\"ahler condition, does not 
reduce to a scalar equation and has been studied in \cite{Ver-Ricci} by the third named author.
Subsequently, the first and the third named authors in \cite{Klaus-Vertman, Klaus-Vertman2} established stability and studied 
solitons and Perelman entropies of this flow in the special case of isolated conical singularities. 
The present work is a continuation of this research program. \medskip

Let us point out that Ricci flow preserving the initial singular structure, is not the only 
possible way to evolve a singular metric. In fact, Giesen and Topping \cite{Topping, Topping2} construct a solution to the Ricci flow on 
surfaces with singularities, which becomes instantaneously complete. Alternatively, Simon \cite{MS} constructs Ricci flow in dimension 
two and three that smoothens out the singularity.

 %%%%%%%%%%%%%%%%%%%%%%%%
\section{Geometric preliminaries on conical manifolds} \medskip
%%%%%%%%%%%%%%%%%%%%%%%%%

We begin with a definition of spaces with isolated conical singularities.
We point out that part of our analysis in fact applies to spaces with non-isolated conical 
singularities, the so-called edges.

\begin{defn}\label{cone-metric}
Let $M$ be the open interior of a compact smooth manifold $\overline{M}$ with boundary 
$F := \partial M$. 
Let $x$ be a boundary defining function and $\ov{\cC(F)}$ a tubular neighborhood of the boundary with 
open interior $\cC(F) := (0,1)_x \times F$. An incomplete Riemannian metric $g$ on $M$ 
with an isolated conical singularity is a smooth metric on $M$ satisfying  
\begin{equation*}
g \restriction \cC(F) = dx^2 + x^2 g_F + h =: \overline{g} + h,
\end{equation*}
where the higher order term $h$ has the following asymptotics 
at $x=0$. Let $\overline{g} = dx^2 + x^2 g_F$ denote the exact conical part of the metric $g$ over 
$\cC(F)$ and $\nabla_{\overline{g}}$ the corresponding Levi Civita connection. 
Then we require that for some $\gamma > 0$ and all integer $k \in \N_0$ the pointwise norm 
\begin{align}\label{higher-order}
| \, x^k \nabla_{\overline{g}}^k h \, |_{\overline{g}} = O(x^\gamma), \ x\to 0. 
\end{align}
\end{defn}

\begin{remark}
	We emphasize here that we do not assume that the higher order term $h$
	is smooth up to $x=0$ and do not restrict the order $\gamma>0$ to be integer. 
	In that sense the notion of conical singularities in the present discussion is more 
	general than the classical notion of conical singularities where $h$ is usually assumed
	to be smooth up to $x=0$ with $\gamma = 1$. This minor generalization is necessary, 
	since the Ricci de Turck flow, which will be introduced below, preserves a conical
    singularity only up to a higher order term $h$ as above.
\end{remark}

We call $(M,g)$ a manifold with an isolated conical singularity, or a \textit{conical manifold} for short. 
The definition extends directly to conical manifolds with finitely many isolated conical singularities. Since the analytic arguments are local 
in nature, we may assume without loss of generality that $M$ has a single conical singularity only. \medskip

Let $(z)=(z_1,\ldots, z_n)$ be local coordinates on $F$, where $n=\dim F$. Then $(x,z)$ are 
local coordinates on the conical neighborhood $\cC(F) \subset M$. A b-vector field is by definition a smooth vector field on 
$\overline{M}$ which is tangent to the boundary $\partial M = F$. The b-vector fields form a Lie algebra, denoted by 
$\V_b$. In the local coordinates $(x,z)$, the algebra $\V_b$ is locally, near $\partial M$, generated by 
\[
\left\{x\frac{\partial}{\partial x},  
\partial_z = \left( \frac{\partial}{\partial z_1},\dots, \frac{\partial}{\partial z_n} \right)\right\},
\]
with smooth coefficients on $\overline{M}$. 
The b-tangent bundle ${}^bTM$ is defined by requiring that the 
b-vector fields form a spanning set of sections for it, i.e. $\mathcal{V}_b=C^\infty(\overline{M},{}^bTM)$. 
The b-cotangent bundle ${}^bT^*M$ is the dual bundle and locally, near $\partial M$, generated by the one-forms
\begin{align}\label{triv}
\left\{\frac{dx}{x}, dz_1,\dots,dz_n\right\}.
\end{align}
Note that the differential form $\frac{dx}{x}$ is singular in the usual sense, 
but smooth as a section of ${}^bT^*M$. Extend the radial function $x$ of the 
cone $\cC(F)$ to a nowhere vanishing smooth function $x: \overline{M} \to [0,2]$.
Then we define the incomplete 
b-tangent space ${}^{ib}TM$ by asking that $C^\infty(\overline{M},{}^{ib}TM) := x^{-1}
C^\infty(\overline{M},{}^{b}TM)$. The dual bundle, the incomplete b-cotangent bundle 
${}^{ib}T^*M$, is related to its complete counterpart ${}^bT^*M$ by 
$C^\infty(\overline{M},{}^{ib}T^*M) = x C^\infty(\overline{M},{}^{b}T^*M)$,
and is locally generated by  
\begin{align}\label{triv2}
\left\{dx, x dz_1,\dots, x dz_n\right\}.
\end{align}

%%%%%%%%%%%%%%%%%%%%%%%%%%%%
\section{Weighted H\"older spaces on conical manifolds}\label{spaces-section}
%%%%%%%%%%%%%%%%%%%%%%%%%%%%

In this section we review definitions from \cite[Section 1.3]{Ver-Ricci} in the case of isolated conical singularities. 
Let $(M,g)$ be a manifold with isolated conical singularities. 

\begin{defn}\label{hoelder-A}
Let $d_M(p,p')$ denote the distance between two points $p,p'\in M$ with respect to the conical metric $g$. In terms of the local coordinates 
$(x,z)$ over the singular neighborhood $\cC(F)$, the distance can be estimated up to a constant uniformly from above and 
below by the distance function of the model cone 
\begin{align*}
d((x,z), (x',z'))=\left(|x-x'|^2+(x+x')^2|z-z'|^2\right)^{\frac{1}{2}}.
\end{align*}
The H\"older space $\ho(M \times [0,T]), \A\in [0,1),$ is by definition the set of functions 
$u(p,t)$ that are continuous on $\overline{M} \times [0,T]$ with finite $\A$-th H\"older 
norm
\begin{align}\label{norm-def}
\|u\|_{\A}:=\|u\|_{\infty} + \sup \left(\frac{|u(p,t)-u(p',t')|}{d_M(p,p')^{\A}+|t-t'|^{\A/2}}\right) <\infty,
\end{align}
with supremum taken over all $p, p' \in M$ and $t,t' \in [0,T]$ with $p \neq p'$ and $t \neq t'$.
\footnote{We can assume without loss of generality that 
the tuples $(p,p')$ are always taken from within the same coordinate 
patch of a given atlas.}
\end{defn} 

We now extend the notion of H\"older spaces to sections of the  
vector bundle $S:=\textup{Sym}^2({}^{ib}T^*M)$ of symmetric $2$-tensors. Note that the 
Riemannian metric $g$ induces a fibrewise inner product on $S$, which we also denote by $g$.

\begin{defn}\label{S-0-hoelder}
The H\"older space $\ho (M \times [0,T], S)$ is by definition the set of all sections $\w$ of
$S$ which are continuous on $\overline{M} \times [0,T]$, 
such that for any local orthonormal frame 
$\{s_j\}$ of $S$, the scalar functions $g(\w,s_j)$ are in $\ho (M \times [0,T])$.
\medskip

The $\A$-th H\"older norm of $\w$ is defined using a partition of unity
$\{\phi_j\}_{j\in J}$ subordinate to a cover of local trivializations of $S$, with a 
local orthonormal frame $\{s_{jk}\}$ over $\supp (\phi_j)$ for each $j\in J$. We put
\begin{align}\label{partition-hoelder-2}
\|\w\|^{(\phi, s)}_{\A}:=\sum_{j\in J} \sum_{k} \| g(\phi_j \w,s_{jk}) \|_{\A}.
\end{align}
\end{defn}

Different choices of $(\{\phi_j\}, \{s_{jk}\})$ lead to equivalent norms so that we may drop the upper index 
$(\phi, s)$ from notation. Next we come to weighted and higher order H\"older spaces of $S$-sections. 

\begin{defn}\label{funny-spaces0} 
	\begin{enumerate}
		\item The (hybrid) weighted H\"older space for $\gamma \in \R$ is
		\begin{align*}
		&\ho(M \times [0,T], S)_{\gamma}  := x^\gamma \ho(M \times [0,T], S)  \, \cap \, 
		x^{\gamma + \A} \mathcal{C}^0_{\textup{ie}}(M \times [0,T], S) \\
		&\textup{with H\"older norm} \  \| \w \|_{\A, \gamma} := \|x^{-\gamma} 
		\w\|_{\A} + \|x^{-\gamma-\A} \w\|_\infty.
		\end{align*}
		\item The weighted higher order H\"older spaces are defined by
		\begin{equation*}
		\begin{split}
		\hok (M \times [0,T], S)_\gamma := \ &\{\w\in \ho_{,\gamma} \mid  \{\V_b^j \circ (x^2 \p_t)^l\} \w \in \ho_{,\gamma} \\ &\textup{for all} \ j+2l \leq k \}, 
		\end{split}
		\end{equation*} 
		for any $\gamma \in \R$ and $k \in \N$. For any $\g > -\alpha$
		and $k \in \N$ we also define
			\begin{equation*}
			\begin{split}
			\hok (M \times [0,T], S)^b_\gamma := \ &\{\w\in \ho \mid  \{\V_b^j \circ (x^2 \p_t)^l\} \w \in \ho_{,\gamma} \\ &\textup{for all} \ j+2l \leq k \}. 
			\end{split}
			\end{equation*} 
		The corresponding H\"older norms are defined 
		using a finite cover of coordinate charts trivializing $S_0$ and a subordinate partition of 
		unity $\{\phi_j\}_{j \in J}$. By a slight abuse of notation, we 
		identify $\V_b$ with its finite family of generators over each coordinate chart. Writing  
		$\mathscr{D} :=\{\V_b^j \circ (x^2 \p_t)^l \mid j+2l \leq k\}$ the H\"older norm on $\hok (M \times [0,T], S)_\gamma$ is then given by
		\begin{equation}\label{higher Hoelder norms}
		\begin{split}
		\|\w\|_{k+\A, \gamma} = \sum_{j\in J} \sum_{X\in \mathscr{D}} \| X (\phi_j \w) \|_{\A, \gamma} + \|\w\|_{\A, \gamma}.
		\end{split}
		\end{equation} 
		For the H\"older norm on $\hok (M \times [0,T], S)^b_\gamma$
	    replace in \eqref{higher Hoelder norms} 
		$\|\w\|_{\A, \gamma}$ by $\|\w\|_{\A}$.
	\end{enumerate}
\end{defn} 

The H\"older norms for different choices of coordinate charts, the subordinate partition of unity or vector fields 
$\V_b$ are equivalent. Analogously we also consider time-independent H\"older spaces, which are 
denoted in the same way with $[0,T]$ deleted from notation above.
\medskip

The vector bundle $S$ decomposes into a direct sum of sub-bundles $S= S_0 \oplus S_1$,
where the sub-bundle $S_0=\textup{Sym}_0^2({}^{ib}T^*M)$
is the space of trace-free (with respect to the fixed metric $g$) symmetric $2$-tensors,
and $S_1$ is the space of pure trace (with respect to the fixed metric $g$) symmetric 
$2$-tensors. Definition \ref{funny-spaces0} extends verbatim to subbundles $S_0$ and $S_1$.

\begin{remark}
	The spaces presented here are slightly different from the spaces originally introduced
	in \cite{Ver-Ricci}. There, in case of $S_1$-sections, higher order weighted 
	H\"older spaces were defined in terms of $x^\g\ho$ instead of $\ho_{,\gamma}$.
	Here, we present a more unified definition, which will become much more convenient 
	below. The arguments of \cite{Ver-Ricci} still carry over to yield regularity statements 
	in these unified H\"older spaces.
\end{remark}

\begin{defn}\label{H-space}
	Let $(M,g)$ be a compact conical manifold, $\g_0,\g_1 \in \R$. In order to simplify 
	notation, we set $\hok (M\times [0,T], S)^b_{\gamma} := \hok (M\times [0,T], S)^b_{\gamma}$
	for $\g \leq -\alpha$, with the same convention for $S_1$ as well. Then we define the following 
	spaces.
	\begin{enumerate}
	\item If $(M,g)$ is not an orbifold, we set 
	\begin{align*} 
	\cH^{k, \A}_{\gamma_0,\gamma_1} (M\times [0,T], S) &:= 
	\hok (M\times [0,T], S_0)_{\gamma_0} \oplus 
	\hok (M\times [0,T], S_1)^b_{\gamma_1}.
	 \end{align*}
	 \item If $(M,g)$ is an orbifold, we set 
	\begin{align*} 
	\cH^{k, \A}_{\gamma_0,\gamma_1} (M\times [0,T], S) :=
	\hok (M\times [0,T], S_0)^b_{\gamma_0} \oplus
	\hok (M\times [0,T], S_1)^b_{\gamma_1}.
	\end{align*}
	\end{enumerate}
\end{defn}

Note that e.g. for the non-orbifold case, the different choice of spaces for the 
$S_0$ and $S_1$ components, simply ensures that 
the $S_1$ component is not (!) included in $x^{\gamma} 
\mathcal{C}^0_{\textup{ie}}$ for some positive weight $\gamma$. 
In case of orbifold singularities, this restriction is imposed on
both the $S_0$ and $S_1$ components. 
We can now impose regularity assumptions on our initial data $(M,g)$.
Definitions \ref{S-0-hoelder} and \ref{funny-spaces0} extend naturally
to associated bundles of ${}^{ib}TM$. Also we write $\cH^{k, \A}_{\gamma} \equiv \cH^{k, \A}_{\gamma,\gamma}$.

\begin{defn}
		Let $\a \in [0,1), k \in \N_0$ and $\g > 0$. A conical manifold $(M,g)$ is $(\a,\g,k)$ H\"older regular if the following two conditions are satisfied:
		\begin{itemize}
			\item[(i)] The $(0,4)$ curvature tensor 
			$ \Rm  \in \hok (M \times [0,T], \otimes^4 \ {}^{ib}T^*M)_{-2}.$ \\[-3mm]
			\item[(ii)] The Ricci curvature tensor 
			$\ric \in \cH^{k, \A}_{-2 + \gamma} (M\times [0,T], S)$. 
		\end{itemize}
	\end{defn}

We continue under that assumption from now on. 

\begin{remark}\label{Einstein}
The asymptotics of the Ricci curvature tensor, as a section of $S$, is generically
$O(x^{-2})$ as $x\to 0$. Hence $(\a,\g,k)$ H\"older regularity with $\gamma > 0$ 
in particular implies that the exact conical part $(\cC(F),\overline{g})$ 
must be Ricci-flat. This is equivalent to $(F^n,g_F)$ being Einstein with 
Einstein constant $n-1$. Moreover, the weight $\gamma$ corresponds to the weight in 
\eqref{higher-order}.
\end{remark}

%%%%%%%%%%%%%%%%%%%%%%%%%%%%
\section{Lichnerowicz Laplacian and tangential stability}
%%%%%%%%%%%%%%%%%%%%%%%%%%%%

Consider the right hand side $-2 \,\ric (g(t)) + \mathcal{L}_{W(t)}g(t)$ 
of the Ricci de Turck flow equation \eqref{RTF}. Write $W(t)= W(g(t),h)$
to indicate precisely what the de Turck vector field depends on. Choose $h=g$
as the background metric. Then, replacing $g(t) = g+s \w$ for some symmetric
$2$-tensor $\w$, the linearization 
of right hand side of \eqref{RTF} is given by 
\begin{equation}
\frac{d}{ds} \left( -2 \, \ric (g+s\w) + \mathcal{L}_{W(g+s\w,g)}(g+s\w) 
\frac{}{} \right) |_{s=0}=-\Delta_{L}\w,
\end{equation}
where $\Delta_L$ is an elliptic operator, 
which is known as the Lichnerowicz Laplacian of $g$, 
acting on symmetric $2$-tensors by
\[ 
\D_L\w_{ij} = \D\w_{ij} - 2 g^{pq} \Rm^r_{qij} \w_{rp} + g^{pq} R_{ip} \w_{qj} + g^{pq}R_{jp} \w_{iq},
\]
where $\D$ denotes the rough Laplacian, and $\Rm^r_{qij}$ and $R_{ij}$ denote the components of the $(1,3)$-Riemann curvature tensor and the Ricci tensor, respectively. 
Near the conical singularity $\D_L$ can be written as follows. 
We choose local coordinates $(x,z)$ over the singular neighborhood 
$\cC(F) = (0,1)_x \times F$. Consider a decomposition of compactly supported smooth sections $C^\infty_0(\cC(F), S \restriction \cC(F))$ 
\begin{equation}
\begin{split}
C^\infty_0(\cC(F), S \restriction \cC(F)) &\to C^\infty_0((0,1), C^\infty(F) 
\times \Omega^1(F) \times \textup{Sym}^2(T^*F)),\\
\w &\mapsto \left( \w(\partial_x, \partial_x), \w (\partial_x, \cdot ), 
\w (\cdot, \cdot)\right),
\end{split}
\end{equation}
where $\Omega^1(F)$ denotes differential $1$-forms on $F$.
Under such a decomposition, the Lichnerowicz Laplace operator 
$\Delta_L$ associated to the singular Riemannian metric $g$ 
attains the following form over $\cC(F)$
\begin{equation}
\Delta_L = - \frac{\partial^2}{\partial x^2} - \frac{n}{x} \frac{\partial}{\partial x}
+ \frac{\square_L}{x^2} + \mathscr{O},
\end{equation}
where $\square_L$ is a differential operator on $C^\infty(F) 
\times \Omega^1(F) \times \textup{Sym}^2(T^*F)$, depending only on the
exact conical part $\overline{g}$, and the 
higher order term is $\mathscr{O} \in x^{-2+\gamma} \V_b^2$ with H\"older regular
coefficients. 

%%%%%%%%%%%%%%%%%%%%%%%%%%%%
\subsection{Tangential stability}\ \medskip
%%%%%%%%%%%%%%%%%%%%%%%%%%%%

We can now impose a central analytic condition, under which existence
of singular Ricci de Turck flow has been established in \cite{Ver-Ricci}. 

\begin{defn}\label{tangential-stability-def}
	$(F,g_F)$ is called (strictly) tangentially stable if the tangential operator 
	$\square_L$ of the Lichnerowicz Laplacian on its cone restricted to tracefree tensors is 
	non-negative (resp. strictly positive).
\end{defn}

Tangential stability (in fact a smaller lower bound $\square_L \geq -((n-1)/2)^2$
would suffice) has a straighforward implication: for $\w \in C^\infty_0(\cC(F), S)$ we find
\begin{equation}\label{lower-bound1}
\begin{split}
\Delta_L  \w =
&\left(  \frac{\partial}{\partial x} + \frac{1}{x} \left( \sqrt{\square_L + \left(\frac{n-1}{2}\right)^2} + \frac{n-1}{2} \right) \right)^t \\
\circ &\left(  \frac{\partial}{\partial x} + \frac{1}{x} \left( \sqrt{\square_L + \left(\frac{n-1}{2}\right)^2} + \frac{n-1}{2} \right) \right)  \w. 
\end{split}
\end{equation}
Thus, tangential stability in particular implies that $\Delta_L$, acting on compactly supported smooth sections, has a lower bound, 
i.e. there exists some $C \in \R$ such that 
\begin{equation}\label{lower-bound}
\Delta_L  \restriction C^{\infty}_0(M,S) \geq C.
\end{equation}
For convenience of the reader we shall add here a complete characterization of
tangential stability, obtained by the first and third named authors in 
\cite{Klaus-Vertman}. Recall that by Remark \ref{Einstein}, the assumption of
$(\a,\g,k)$ H\"older regularity implies that $(F,g_F)$ is Einstein with 
Einstein constant $(n-1)$.

\begin{thm}
	Let $(F,g_F)$, $n\geq 3$ be a compact Einstein manifold with constant $(n-1)$. 
	We write $\Delta_E$ for its Einstein operator, and denote the Laplace Beltrami 
	operator by $\Delta$. Then $(F,g_F)$ is tangentially stable if and only if 
	$\mathrm{Spec}(\Delta_E|_{TT})\geq0$ and $\mathrm{Spec}(\Delta)\setminus 
	\left\{0\right\}\cap (n,2(n+1))=\varnothing$. Similarly, $(M,g)$ is strictly tangentially 
	stable if and only if $\mathrm{Spec}(\Delta_E|_{TT})>0$ and $\mathrm{Spec}
	(\Delta)\setminus \left\{0\right\}\cap [n,2(n+1)]=\varnothing$. 
\end{thm}

There are plenty of examples, where (strict) tangential stability is satisfied. 
Any spherical space form is tangentially stable. \cite{Klaus-Vertman} also provides a
detailed list of strict tangentially stable Einstein manifolds that are symmetric spaces. 
Note that $\mathbb{S}^n$ is tangentially stable but not strictly tangentially stable. 

\begin{thm}
	Let $(F^n,g_F)$, $n\geq 2$ be a closed Einstein manifold with constant $(n-1)$, 
	which is a symmetric space of compact type. If it is a simple Lie group $G$, 
	it is strictly tangentially stable if $G$ is one of the following spaces:
	\begin{align}
	\mathrm{Spin}(p)\text{ }(p\geq 6,p\neq 7),\qquad \mathrm{E}_6,
	\qquad\mathrm{E}_7,\qquad\mathrm{E}_8,\qquad \mathrm{F}_4.
	\end{align}
	If the cross section is a rank-$1$ symmetric space of compact type $G/K$, 
	$(M,g)$ is strictly tangentially stable if  $G$ is one of the following real Grassmannians
	\begin{equation}
	\begin{aligned}
	&\frac{\mathrm{SO}(2q+2p+1)}{\mathrm{SO}(2q+1)\times \mathrm{SO}(2p)}\text{ }(p\geq 2,q\geq 1),\qquad
	\frac{\mathrm{SO}(8)}{\mathrm{SO}(5)\times\mathrm{SO}(3)},\\
	&\frac{\mathrm{SO}(2p)}{\mathrm{SO}(p)\times \mathrm{SO}(p)}\text{ }(p\geq 4),\qquad
	\frac{\mathrm{SO}(2p+2)}{\mathrm{SO}(p+2)\times \mathrm{SO}(p)}\text{ }(p\geq 4)\\
	&\frac{\mathrm{SO}(2p)}{\mathrm{SO}(2p-q)\times \mathrm{SO}(q)}\text{ }(p-2\geq q\geq 3),
	\end{aligned}
	\end{equation}
	or one of the following spaces:
	\begin{equation}
	\begin{aligned}
	\mathrm{SU}(2p)/\mathrm{SO}(p)\text{ }(n\geq 6),\qquad
	&\mathrm{E}_6/[\mathrm{Sp}(4)/\left\{\pm I\right\}],\qquad \quad
	\mathrm{E}_6/\mathrm{SU}(2)\cdot \mathrm{SU}(6),\\
	\mathrm{E}_7/[\mathrm{SU}(8)/\left\{\pm I\right\}],\qquad&
	\mathrm{E}_7/\mathrm{SO}(12)\cdot\mathrm{SU}(2),\qquad
	\mathrm{E}_8/\mathrm{SO}(16),\\
	\mathrm{E}_8/\mathrm{E}_7\cdot \mathrm{SU}(2),\qquad&
	\mathrm{F}_4/Sp(3)\cdot\mathrm{SU}(2).
	\end{aligned}
	\end{equation}
\end{thm}

%%%%%%%%%%%%%%%%%%%%%%%%%%%%
\subsection{Self-adjoint extensions of the Lichnerowicz Laplacian}\ \medskip
%%%%%%%%%%%%%%%%%%%%%%%%%%%%

We can now study the self-adjoint closed extensions of $\Delta_L$
acting on compactly supported sections $C^{\infty}_0(M,S)$.
We write $L^2(M,S)$ for the completion of $C^{\infty}_0(M,S)$ with respect to the  
$L^2$-norm $( \cdot, \cdot )_{L^2}$ defined by $g$. The maximal closed extension 
of $\Delta_L$ in $L^2(M,S)$ is defined by the following domain
\begin{equation}
\dom(\Delta_{L, \max}) := \{\w \in L^2(M,S) \mid \Delta_L \w \in L^2(M,S)\},
\end{equation}
where $\Delta_L \w$ is defined distributionally. 
The minimal closed extension of $\Delta_L$ in $L^2(M,S)$ is obtained as the domain of 
the graph closure of $\Delta_L$ acting on $C^{\infty}_0(M,S)$
\begin{equation*}
\begin{split}
\dom(\Delta_{L, \min}) := \{\w \in \dom(\Delta_{L, \max}) \mid \exists (\w_n)_{n\in \N} \subset C^{\infty}_0(M,S): \\
\w_n \xrightarrow{n\to \infty} \w, \quad \Delta_L \w_n \xrightarrow{n\to \infty} \Delta_L \w \ \textup{in} \ L^2(M,S)\}.
\end{split}
\end{equation*}

Let $(\lambda, \w_\lambda)$ be the set of eigenvalues and 
corresponding eigentensors of the tangential operator $\square_L$.
Assuming tangential stability, we have $\lambda \geq 0$, and hence we define
\begin{equation}\label{nu}
\nu(\lambda) := \sqrt{\lambda + 
	\left(\frac{n-1}{2}\right)^2}.
\end{equation}
Standard arguments, see e.g. \cite{KLP:FDG}, 
show that for each $\w \in \dom(\Delta_{L, \max})$
there exist constants $c^\pm_\lambda, \nu(\lambda) \in [0,1)$, 
depending only on $\w$, such that $\w$ admits a partial asymptotic expansion as $x\to 0$
\begin{equation}\label{cone-asymptotics}
\begin{split}
\w & = \sum_{\nu(\lambda) = 0} \left(c^+_{\lambda}(\w) x^{ - \frac{(n-1)}{2}}
+ c^-_{\lambda}(\w) x^{ - \frac{(n-1)}{2}} \log(x) \right) \cdot \omega_\lambda 
\\ & + \sum_{\nu(\lambda) \in (0,1)} \left(c^+_{\lambda}(\w) x^{\nu(\lambda) - \frac{(n-1)}{2}}
+ c^-_{\lambda}(\w) x^{-\nu(\lambda) - \frac{(n-1)}{2}} \right) \cdot \omega_\lambda 
\\ & + \widetilde{\w}, \quad \widetilde{\w} \in \dom(\Delta_{L, \min}).
\end{split}
\end{equation}
All self-adjoint extensions for $\Delta_L$ can be classified by 
algebraic conditions on the coefficients in the asymptotic 
expansion \eqref{cone-asymptotics}, see e.g. Kirsten, 
Loya and Park \cite[Proposition 3.3]{KLP:FDG}. The Friedrichs
self-adjoint extension of $\Delta_L$ on $C^\infty_0(M,S) \subset L^2(M,S)$ 
is given by the domain
\begin{equation}\label{Friedrichs-domain}
\dom(\Delta_{L}) := \{\w \in \dom(\Delta_{L, \max}) \mid 
c^-_{\lambda}(\w) = 0 \ \textup{for} \ \nu(\lambda) \in [0,1)\}.
\end{equation}
Note that if $n\geq 3$, then tangential stability implies that 
all $\nu(\lambda) \geq 1$. Hence the minimal and maximal domains coincide and
we find
\begin{equation}\label{Friedrichs-domain}
\dom(\Delta_{L}) = \dom(\Delta_{L, \max}) = \dom(\Delta_{L, \min}),
\quad \textup{for} \ n\geq 3.
\end{equation}

We close the section with an observation by Friedrichs and 
Stone, see Riesz and Nagy \cite[Theorem on p. 330]{RN}, cf. 
the corresponding statement in our previous work \cite[Proposition 2.2]{Klaus-Vertman}.

\begin{prop}\label{Friedrichs}
	Assume that $(M,g)$ is tangentially stable, so that by \eqref{lower-bound} the Lichnerowicz 
	Laplacian $\Delta_L$ with domain $C^\infty_0(M,S)$ is bounded from 
	below by a constant\footnote{The case of $C=0$ is commonly referred 
		to as linear stability in the literature.} $C \in \R$. 
	Then the Friedrichs self-adjoint extension of the Lichnerowicz Laplacian 
	$\Delta_L$ is bounded from below by $C$ as well.
\end{prop}

%%%%%%%%%%%%%%%%%%%%%%%%%%%%
\section{Existence and regularity of singular Ricci de Turck flow}
%%%%%%%%%%%%%%%%%%%%%%%%%%%%

The main result of \cite[Theorem 4.1]{Ver-Ricci}, see also 
\cite[Theorem 1.2]{Klaus-Vertman}, is existence and regularity 
of singular Ricci de Turck flow. Despite a slight difference of H\"older spaces
used here and in \cite{Ver-Ricci}, we still conclude from  
\cite[Theorem 4.1]{Ver-Ricci} that the heat operator of the Friedrichs extension
$\D_L$ maps
$$
e^{-t\D_L}: \cH^{k, \A}_{-2+\gamma_0,-2+\gamma_1}(M \times [0,T], S) 
\to \cH^{k+2, \A}_{\gamma_0,\gamma_1}(M \times [0,T], S).
$$
Therefore existence and regularity obtained in 
\cite[Theorem 4.1]{Ver-Ricci}, as well as in
\cite[Theorem 1.2]{Klaus-Vertman} for a different choice of
a background metric, still hold in our spaces.

\begin{thm}\label{existence}
	Let $(M,g)$ be a conical manifold, which is tangentially stable and 
	$(\alpha, k+1,\gamma)$ H\"older regular. In case the conical singularity 
	is not orbifold, we assume strict tangential stability.
	Let the background metric be either equal to $g$ or a conical Ricci flat metric, in which case $g_0$ is assumed to be a sufficiently small perturbation of $\widetilde{g}$ in $\cH^{k+2, \A}_{\gamma,\g} (M, S)$.
	Then there exists some $T>0$, such that the Ricci de Turck flow \eqref{RTF}, 
	starting at $g$ admits a solution $g(\cdot) \in \cH^{k+2, \A}_{\gamma_0,\gamma_1} 
	(M \times [0,T], S)$ for some $\gamma_0,\gamma_1 \in (0,\gamma)$ sufficiently small.
\end{thm}

Let us now explain in what sense the flow preserves the conical singularity.
Given an admissible perturbation $g$ of the conical metric $g_0$, the pointwise trace of $g$ 
with respect to $g_0$, denoted as $\textup{tr}_{g_0} g$ is by definition of admissibility an element of the 
H\"older space $\hok (M, S_1)^b_{\gamma}$, restricting at $x=0$ to a constant function $(\textup{tr}_{g_0} g)(0) = u_0>0$.
Setting $\widetilde{x} := \sqrt{u_0} \cdot x$, the admissible perturbation $g= g_0 + h$
attains the form 
$$g = d\widetilde{x}^2 + \widetilde{x}^2 g_F + \widetilde{h},$$
where $|\widetilde{h}|_g = O(x^\gamma)$ as $x\to 0$. Note that the leading part of the 
admissible perturbation $g$ near the conical singularity differs from the leading part of the admissible metric $g_0$ only by scaling. \medskip

We now wish to specify conditions on the weights $\gamma_0,\gamma_1$. 
Let us write $\square'_L$ for the tangential operator of the Lichnerowicz Laplacian
$\D_L$ acting on trace-free tensors $S_0$. The tangential operator of $\D_L$ acting on the
pure trace tensors $S_1$ is simply the Laplace Beltrami operator $\Delta_F$ of $(F,g_F)$.
We set $u_0:=\min(\Spec \, \square'_L \backslash\{0\})>0$ and $u_1:= \min(\Spec \, \D_F \backslash\{0\}) >0$. In the orbifold case we set $u:= \min \{u_0,u_1\}$. Define
\footnote{cf. \cite[Remark 2.5]{Ver-Ricci}}
\[ \mu_0 := \sqrt{u_0 + \frac{n-1}{2}} - \frac{n-1}{2}, \qquad \mu_1 := \sqrt{u_1 + \frac{n-1}{2}} - \frac{n-1}{2}. \]
Then the admissible choice of weights $(\gamma_0,\gamma_1)$
(in the orbifold case set $\gamma= \gamma_0 = \gamma_1$) and the H\"older exponent
$\alpha$ is given by the following restrictions. 

\begin{equation}\label{gamma01}
\begin{split}
& i) \quad \g_0 \in (0,\mu_0), \qquad \g_0 \le 2 \g_1, \qquad \g_0 < \g, \\
& ii) \quad \g_1 \in (0,\mu_1), \qquad \g_1 \le \g_0, \qquad \g_1 < \g, \\
& iii) \quad \a \in (0,\mu_0 - \g_0) \cap (0,\mu_1 - \g_1).
\end{split}
\end{equation}

\begin{rem}\label{DL1}
Note that in case of $u_0, u_1 > n$, i.e. $\square_L > n$, and assuming additionally 
$\gamma > 1$, we may choose $\g_0, \g_1 > 1$ satisfying \eqref{gamma01}. 
This stronger condition is studied in the Appendix \S \ref{strong-tangential},
where a list of examples is provided.
\end{rem}

We close the section by pointing out regularity of 
the de Turck vector field. As before we may define weighted 
higher order H\"older spaces of the incomplete $b$-cotangent bundle 
${}^{ib}T^*M$. By Theorem \ref{existence}, cf. \cite[\S 6]{Ver-Ricci}
\begin{align}\label{W-regularity} 
W(t) \in \mathcal{C}^{k+1, \A}_{\textup{ie}}(M \times [0,T], 
{}^{ib}T^*M)_{-1+\ov{\g}},
\end{align}
where $\ov{\g} := \min\{\g_0,\g_1\}$, and 
$\ov{\g}=\gamma$ in the orbifold case. 
As explained in the previous Remark \ref{DL1}, 
assuming $\square_L > n$ we can choose $\ov{\g}>1$, so that 
existence time of the integral curves of $W(t)$ is 
positive, uniformly bounded away from zero. In this case we may
pass from the Ricci de Turck flow back to Ricci flow, which 
is generally not clear in the singular setting.

%%%%%%%%%%%%%%%%%%%%%%%%%%%%
\section{Evolution of curvatures under Ricci de Turck flow} \label{Evolution of the scalar curvature under Ricci de Turck flow} \medskip
%%%%%%%%%%%%%%%%%%%%%%%%%%%%

The results of this section are well-known in the setting of smooth compact manifolds where the Ricci flow is defined. Evolution equations are usually proven by studying evolution of curvatures under the Ricci flow, and then the corresponding equations follow for the Ricci de Turck flow by applying the corresponding diffeomorphisms. In the singular setting there might be no globally well-defined diffeomorphism to go from the Ricci de Turck flow back to the Ricci flow, as the de Turck vector field may be pointing toward the singularity. Thus we need to establish the evolution of Ricci curvature under Ricci de Turck flow directly without passing back to the Ricci flow. The evolution of the scalar curvature is then a direct consequence.\medskip

\textit{Notation:} Let $g(t), t \in [0,T]$ be a Ricci de Turck flow of Riemannian metrics on $M$. We denote by $\n$, $\D$, $\D_L$ and $\G$ the covariant derivative, the Laplace Beltrami operator, the Lichnerowicz Laplacian and the Christoffel symbols, respectively, defined with respect to the metric $g(t)$. We let $|\cdot|$ be the norm with respect to $g(t)$. We denote the Riemann curvature tensor by $\Rm$, the Ricci tensor by $\ric$ and the scalar curvature by $R$, and we use the following conventions:
\[ \Rm(X,Y)Z := \n_X (\n_Y Z) - \n_Y (\n_X Z) - \n_{[X,Y]} Z, \]
\[ Rm(\frac{\p}{\p x^i}, \frac{\p}{\p x^j})\frac{\p}{\p x^k} = R^l_{ijk} \frac{\p}{\p x^l}. \]
Here and below we use the Einstein summation convention and sum over repeated upper and lower indices. We lower the upper index to the fourth slot
$R_{ijkl} := g_{pl} R^p_{ijk}$. Then the Ricci tensor $\ric$ 
and the scalar curvature $R$ are given by
\[ \ric_{jk}\equiv R_{jk} := R^i_{ijk} = g^{il} R_{ijkl}, \quad R := g^{jk}R_{jk}. \]

\begin{thm}
Let $M$ be a smooth manifold. Let $g(t)$, $t \in [0,T]$ be a solution of the Ricci de Turck flow with initial metric $g$ and reference metric $h$, i.e.
\begin{equation}\label{Ricci de Turck system}
\left\{\begin{array}{c}
\frac{\p}{\p t} g_{ij}(x,t) = 
\left(-2 R_{ij} + \n_i W_j + \n_j W_i\right)(x,t), \qquad  (x,t) \in M \times [0,T],  \\
g_{ij}(x,0) = g_{ij}(x), \qquad x \in M,
\end{array}\right.
\end{equation}
where $W(t)^k = g(t)^{ij} \left(\Gamma^k_{ij}(g(t)) - \Gamma^k_{ij}(h)\right)$ is the de Turck vector field. Then the Ricci tensor $\ric$ evolves by
\begin{equation}\label{evol equ Ricci}
(\p_t + \D_L) R_{jk} = \n_m R_{jk} W^m + R_{jm} \n_k W^m + R_{km} \n_j W^m.
\end{equation}
\end{thm}

\begin{proof}
By \cite[Lemma 3.5]{ChKn} the evolution of the Ricci curvature is given by
\begin{equation} \label{Ricci1}
\p_t R_{jk} = \frac{1}{2}g^{pq}(\n_q \n_j a_{kp} + \n_q \n_k a_{jp} - \n_p \n_q a_{jk} - \n_j \n_k a_{pq}),
\end{equation}
where the time-dependent $(0,2)$-tensor $a$ is the time-derivative of the metric, $a_{ij} := \frac{\p}{\p t} g_{ij}$. Using the Ricci de Turck flow equation $\frac{\p}{\p t}g_{ij} = -2R_{ij} + \n_i W_j + \n_j W_i$ we obtain
\begin{align}\label{Ricci2}
\p_t R_{jk} = & -g^{pq} (\n_q \n_j R_{kp} + \n_q \n_k R_{jp} - \n_p \n_q R_{jk} - \n_j \n_k R_{pq}) \\
& + \frac{1}{2}g^{pq} (\n_q \n_j \n_k W_p + \n_q \n_j \n_p W_k + \n_q \n_k \n_j W_p + \n_q \n_k \n_p W_j \nonumber \\ & \hspace{1.5cm}- \n_p \n_q \n_j W_k - \n_p \n_q \n_k W_j - \n_j \n_k \n_p W_q - \n_j \n_k \n_q W_p). \nonumber 
\end{align}
The first line on the right-hand side is the time-derivative of the Ricci tensor when the metric evolves by Ricci flow. Since it is known \cite[Lemma 6.9]{ChKn} that under Ricci flow the time-derivative of the Ricci tensor equals the Lichnerowicz Laplacian of the Ricci tensor, i.e. $\p_t R_{jk} = -\D_L R_{jk}$, it follows that 
\begin{equation}\label{Ricci3}
\D_L R_{jk} = g^{pq} (\n_q \n_j R_{kp} + \n_q \n_k R_{jp} - \n_p \n_q R_{jk} - \n_j \n_k R_{pq}).
\end{equation}
To simplify the second and third line on the right-hand side, we compute several terms by commuting covariant derivatives. We have
\begin{align}\label{Ricci4}
\n_p \n_j \n_p W_k & = \n_p (\n_q \n_j W_k - R^m_{jqk} W_m) \\
& = \n_p \n_q \n_j W_k - \n_p R^m_{jqk} W_m - R^m_{jqk} \n_p W_m. \nonumber
\end{align} 
Hence by exchanging $j$ and $k$
\begin{equation} \label{Ricci5}
\n_p \n_k \n_p W_j = \n_p \n_q \n_k W_j - \n_p R^m_{kqj} W_m - R^m_{kqj} \n_p W_m. 
\end{equation} 
Furthermore,
\begin{align}\label{Ricci6}
\n_p \n_j \n_k W_q & = \n_j \n_p \n_k W_q - R^m_{pjk} \n_m W_q - R^m_{pjq} \n_k W_m \\
& = \n_j (\n_k \n_p W_q - R^m_{pkq} W_m) - R^m_{pjk} \n_m W_q - R^m_{pjq} \n_k W_m  \nonumber \\
& = \n_j \n_k \n_p W_q - \n_j R^m_{pkq} W_m - R^m_{pkq} \n_j W_m - R^m_{pjk} \n_m W_q - R^m_{pjq} \n_k W_m. \nonumber
\end{align}
By exchanging $j$ and $k$ and commuting covariant derivatives we have
\begin{align}\label{Ricci7}
\n_p \n_k \n_j W_q & = \n_k \n_j \n_p W_q - \n_k R^m_{pjq} W_m - R^m_{pjq} \n_k W_m - R^m_{pkj} \n_m W_q - R^m_{pkq} \n_j W_m \nonumber \\
& = \n_j \n_k \n_p W_q - R^m_{kjp} \n_m W_q - R^m_{kjq} \n_p W_m \\
& \hspace{0.5cm} - \n_k R^m_{pjq} W_m - R^m_{pjq} \n_k W_m - R^m_{pkj} \n_m W_q - R^m_{pkq} \n_j W_m. \nonumber 
\end{align}
Now plugging \eqref{Ricci3}, \eqref{Ricci4}, \eqref{Ricci5}, \eqref{Ricci6}, \eqref{Ricci7} into \eqref{Ricci2} we obtain
\begin{align}\label{Ricci8}
\p_t R_{jk} & = -\D_L R_{jk} - \frac{1}{2} g^{pq} (\n_p R^m_{jqk} W_m + R^m_{jqk} \n_p W_m + \n_p R^m_{kqj} W_m + R^m_{kqj} \n_p W_m \nonumber\\
& \hspace{3.0cm}+ \n_j R^m_{pkq} W_m + R^m_{pkq} \n_j W_m + R^m_{pjk} \n_m W_q + R^m_{pjq} \n_k W_m \nonumber\\
& \hspace{3.0cm}+ R^m_{kjp} \n_m W_q + R^m_{kjq} \n_p W_m + \n_k R^m_{pjq} W_m + R^m_{pjq} \n_k W_m \nonumber\\
& \hspace{3.0cm}+ R^m_{pkj} \n_m W_q + R^m_{pkq} \n_j W_m). 
\end{align}
Simplifying and rearranging terms yields
\begin{align}\label{Ricci9}
\p_t R_{jk} & = -\D_L R_{jk} -\frac{1}{2}g^{pq} \n_q W_m (R^m_{jpk} + R^m_{kpj} + R^m_{kjp}) -\frac{1}{2} g^{pq} \n_m W_q (R^m_{pjk} + R^m_{kjp} + R^m_{pkj}) \nonumber\\
& \hspace{0.5cm}-\frac{1}{2}g^{pq} W_m (\n_p R^m_{jqk} + \n_p R^m_{kqj} + \n_j R^m_{pkq} + \n_k R^m_{pjq})\\
& \hspace{0.5cm}-g^{pq}(R^m_{pkq} \n_j W_m + R^m_{pjq} \n_k W_m). \nonumber
\end{align}
The second line on the right-hand side is 
\begin{align}\label{Ricci10}
& -\frac{1}{2}g^{pq} \n_q W_m (R^m_{jpk} + R^m_{kpj} + R^m_{kjp}) -\frac{1}{2} g^{pq} \n_m W_q (R^m_{pjk} + R^m_{kjp} + R^m_{pkj}) \\
= & -\frac{1}{2} \n^p W^m (R_{jpkm} + R_{kpjm} + R_{kjpm}) -\frac{1}{2} \n^m W^p (R_{pjkm} + R_{kjpm} + R_{pkjm}) \nonumber \\
= & -\frac{1}{2} \n^p W^m (R_{jpkm} + R_{kpjm} + R_{kjpm}) -\frac{1}{2} \n^p W^m (R_{mjkp} + R_{kjmp} + R_{mkjp}) \nonumber\\
= & -\frac{1}{2} \n^p W^m (R_{jpkm} + R_{kpjm} + R_{kjpm} - R_{kpjm} - R_{kjpm} - R_{jpkm}) \nonumber\\
= & \; 0,\nonumber
\end{align}
where in the penultimate line we used the symmetries of the curvature tensor.
The fourth line on the right-hand side is given by 
\begin{equation}\label{Ricci11}
-g^{pq}(R^m_{pkq} \n_j W_m + R^m_{pjq} \n_k W_m) = R_{km} \n_j W^m + R_{jm} \n_k W^m.
\end{equation} 
To compute the third line on the right-hand side, we first observe that by the symmetries of the curvature tensor and the second Bianchi identity
\begin{equation}\label{Ricci12}
\n_p R_{jqkm} = \n_p R_{kmjq} = -\n_m R_{pkjq} - \n_k R_{mpjq}.
\end{equation}
Moreover, exchanging $j$ and $k$ yields
\begin{equation}\label{Ricci13}
\n_p R_{kqjm} = -\n_m R_{pjkq} - \n_j R_{mpkq}.
\end{equation}
Now using \eqref{Ricci12} and \eqref{Ricci13} we obtain for the third line on the right-hand side 
\begin{align}\label{Ricci14}
& -\frac{1}{2}g^{pq} W_m (\n_p R^m_{jqk} + \n_p R^m_{kqj} + \n_j R^m_{pkq} + \n_k R^m_{pjq}) \\
= & -\frac{1}{2}g^{pq} W^m (-\n_m R_{pkjq} - \n_k R_{mpjq} -\n_m R_{pjkq} - \n_j R_{mpkq} + \n_j R_{pkqm} + \n_k R_{pjqm}) \nonumber\\
= & -\frac{1}{2} W^m (-\n_m R_{kj} + \n_k R_{mj} - \n_m R_{jk} + \n_j R_{mk} - \n_j R_{km} - \n_k R_{jm})\nonumber\\
= & \; W^m \n_m R_{jk}.\nonumber
\end{align}
Plugging \eqref{Ricci10}, \eqref{Ricci11} and \eqref{Ricci14} into \eqref{Ricci9} we finally obtain the claimed evolution equation of the Ricci tensor
\begin{equation*}
\p_t R_{jk} = -\D_L R_{jk} + \n_m R_{jk} W^m + R_{jm} \n_k W^m + R_{km} \n_j W^m.
\end{equation*}
\end{proof}

\begin{cor}\label{main theo}
Let $M$ be a smooth manifold. Let $g(t)$, $t \in [0,T]$ be a solution of the Ricci de Turck flow with initial metric $g$ and reference metric $h$, i.e.
\begin{equation}\label{Ricci de Turck system}
\left\{\begin{array}{c}
\frac{\p}{\p t} g_{ij}(x,t) = 
\left(-2 R_{ij} + \n_i W_j + \n_j W_i\right)(x,t), \qquad  (x,t) \in M \times [0,T],  \\
g_{ij}(x,0) = g_{ij}(x), \qquad x \in M,
\end{array}\right.
\end{equation}
where $W(t)^k = g(t)^{ij} \left(\Gamma^k_{ij}(g(t)) - \Gamma^k_{ij}(h)\right)$ is the de Turck vector field. Then the scalar curvature $R$ evolves by
\begin{equation}\label{evol equ}
(\p_t + \D) R = \la W, \n R \ra + 2 |\ric|^2.
\end{equation}
\end{cor}

\begin{proof}
The scalar curvature is given by $R = g^{jk} R_{jk}$. Since 
\[ \p_t g^{jk} = -g^{jp}g^{kq} \p_t g_{pq} = -g^{jp}g^{kq} (-2R_{pq} + \n_p W_q + \n_q W_p) \] 
and by the evolution equation for the Ricci tensor \eqref{evol equ Ricci} we thus obtain
\begin{align*}
\p_t R & = -g^{jp}g^{kq}(-2R_{pq} + \n_p W_q + \n_q W_p) R_{jk} \\
& \hspace{0.5cm}+ g^{jk} (-\D_L R_{jk} + \n_m R_{jk} W^m + R_{jm} \n_k W^m + R_{km} \n_j W^m) \\
& = 2|\ric|^2 -2g^{jp}g^{kq} \n_p W_q R_{jk} - g^{jk} \D_L R_{jk} + \n_m R \, W^m + 2 g^{jk} R_{jm} \n_k V^m \\
& = -\D R + \n_m R \, W^m + 2|\ric|^2,
\end{align*}
where in the last step we used $g^{jk} \D_L R_{jk} = \D R$.
\end{proof}
\begin{rem}
Note that in the proof we didn't use the special form of $W$, we just used that 
$W$ is a (time-dependent) one-form.
\end{rem}

%%%%%%%%%%%%%%%%%%%%%
\section{Regularity of Ricci curvature along the Ricci de Turck flow} \medskip
%%%%%%%%%%%%%%%%%%%%%

Our aim in this section is to improve the a priori low regularity of the Ricci curvature along the singular Ricci 
de Turck flow, as noted in \cite[Theorem 8.1]{Ver-Ricci}. 
Consider an $(\A, k+1, \gamma)$-H\"older regular conical manifold 
$(M,g)$, satisfying tangential stability, with singular Ricci de Turck flow
$g(\cdot ) \in \cH^{k, \A}_{\gamma_0,\gamma_1} (M\times [0,T], S)$, i.e.
decomposing $g(t)= (1+u)g + \w$ into trace and trace-free parts with respect to the initial 
metric $g$, we have
	\begin{align*}
	(\w, u) \in 
	\mathcal{C}^{k+2,\A}_{\textup{ie}}(M\times [0,T], S_0)_{\gamma_0}
	\oplus \mathcal{C}^{k+2,\A}_{\textup{ie}}(M\times [0,T], S_1)^b_{\gamma_1}.
	\end{align*}
	Recall the following transformation rule for the Ricci curvature tensor under 
	conformal transformations (setting $1+u = e^{2\phi}$ and noting that $\dim M = n+1$)
	\begin{equation}\label{conformal}
	\begin{split}
	\ric((1+u)g) = \ric(g) &-(n-1)
	\left( \nabla \partial \phi - \partial \phi \cdot \partial \phi \right)
	\\ &+ (\Delta_L \phi - (n-1) \|\nabla \phi\|^2 ) g.
	\end{split}
	\end{equation}
From here we conclude that $\ric ((1+u)g) - R(g) \in 
\mathcal{C}^{k,\a}_{\textup{ie}}(M \times [0,T],S)_{-2+\g_1}$.
Now consider $\ric((1+u)g + \w) - \ric((1+u)g)$, which is an intricate combination of $a$ and $\w$, 
involving their second order $x^{-2}\V^2$ derivatives. Hence that difference lies in $\mathcal{C}^{k,\a}_{\textup{ie}}(M \times [0,T],S)_{-2+\ov{\g}}$
with $\ov{\g}:= \min\{\gamma_0, \gamma_1\}$. We conclude
\begin{equation}\label{Ricci-regularity}
\begin{split}
&\ric(g) \in \mathcal{C}^{k+1,\a}_{\textup{ie}}(M \times [0,T],S)_{-2+\g}, \\
&\ric(g(t)) \in \mathcal{C}^{k,\a}_{\textup{ie}}(M \times [0,T],S)_{-2+\ov{\g}}, \ t> 0.
\end{split}
\end{equation}
In particular, e.g. if $\mu_0,\mu_1 \leq 2 \leq \gamma$, then $\ov{\g} \leq 2 \leq \gamma$.
In that case, the initial Ricci curvature $\ric (g)$ is bounded as a section of $S$, while
for positive times $\ric (g(t))$ is singular as a section of $S$. In this section we improve
this low regularity result.

%%%%%%%%%%%%%%%%%%%%%
\subsection{Expansion of the Lichnerowicz Laplacian} \medskip
%%%%%%%%%%%%%%%%%%%%%

Below we fix the following notation: 
A supscript " $\widetilde{}$ " indicates that the quantity is taken with respect to the initial metric $g(0)$. 
For example, $\wt{\n}$ and $\wt{\D_L}$ refer to the covariant derivative and the Lichnerowicz Laplacian with respect to $g(0)$. Otherwise the quantities
are defined with respect to the flow $g(t)$.

\begin{thm} \label{L expansion}
Consider an $(\A, k+1, \gamma)$-H\"older regular conical manifold 
$(M,g)$, satisfying tangential stability, with singular Ricci de Turck flow
$g(\cdot ) \in \cH^{k+2, \A}_{\gamma_0,\gamma_1} (M\times [0,T], S)$, i.e.
decomposing $g(t)= ag + \w$ into trace and trace-free parts with respect to the initial 
metric $g$. Then there exist 
\begin{align*}
&A \in \mathcal{C}^{k,\a}_{\textup{ie}}(M \times [0,T],S)_{\ov{\g}}, \\
&B \in \mathcal{C}^{k,\a}_{\textup{ie}}(M \times [0,T],S)_{-1+\ov{\g}},\\
&C \in \mathcal{C}^{k,\a}_{\textup{ie}}(M \times [0,T],S)_{-2+\ov{\g}},
\end{align*}
such that for any (time-dependent) symmetric $2$-tensor $c$ 
	\begin{align}\label{Lich expansion}
		\D_L c  = a^{-1/2} \wt{\D} \, a^{-1/2} c + 
		\left(A \cdot (x^{-1}\V_b)^2 + B \cdot x^{-1}\V_b + C \right) c.
	\end{align}
\end{thm}

\begin{proof}
	Consider the rough Laplacian $\Delta = -g^{ij} \n_i \n_j$ first. We start by writing out $\D$ in local coordinates and later on turn to the Lichnerowicz Laplacian $\Delta_L$. We compute for any symmetric $2$-tensor $c$ 
	\begin{equation} \label{Laplacian} \begin{split}
		&(\D c)_{kl} = -g^{ij} \n_i \n_j c_{kl} \\
		& = -g^{ij} (\p_i \p_j c_{kl} - \p_i \G^m_{jk} c_{ml} - \G^m_{jk} \p_i c_{ml} - \p_i \G^m_{jl} c_{km} - \G^m_{jl} \p_i c_{km} \\
		& - \G^m_{ij} \p_m c_{kl} + \G^m_{ij} \G^p_{mk} c_{pl} + \G^m_{ij} \G^p_{ml} c_{kp} - \G^m_{ik} \p_j c_{ml} + \G^m_{ik} \G^p_{jm} c_{pl}\\
		& + \G^m_{ik} \G^p_{jl} c_{mp} - \G^m_{il} \p_j c_{km} + \G^m_{jk} \G^p_{il} c_{mp} + \G^m_{il} \G^p_{jm} c_{kp}). 
	\end{split}\end{equation}
	Our goal is to obtain an expansion for each term on the right-hand side of \eqref{Laplacian}. The metric $g(t)$ has the form $g_{ij} = a\wt{g}_{ij} + \w_{ij}$,
	with inverse (cf. \cite[p. 28]{Ver-Ricci}) 
	\begin{equation} \label{inverse metric}
		g^{ij} = a^{-1} \wt{g}^{ij} - a^{-2} \wt{g}^{il} \wt{g}^{jp} \w_{pl} + a^{-2} g^{jl} \wt{g}^{ir}\wt{g}^{pq} \w_{lp} \w_{rq}. 
	\end{equation} 
	Hence the first term on the right-hand side of \eqref{Laplacian} is given by
	\begin{align*} 
		-g^{ij} \p_i \p_j c_{kl}  = -a^{-1} \wt{g}^{ij} \p_i \p_j c_{kl}
		-(-a^{-2}\wt{g}^{is}\wt{g}^{jp} \w_{ps} + a^{-2} g^{js} \wt{g}^{ir} \wt{g}^{pq} \w_{sp} \w_{rq}) \p_i \p_j c_{kl}.
	\end{align*}
	Now we study the asymptotics of the last expression at the conical singularity. 
	This is analogous to the discussion in \cite[p. 28-29]{Ver-Ricci}. Consider the 
	local coordinates $(z_0,\cdots, z_n)$ of $\cC(F)$ with $z_0=x$ and 
	$(z_1, \cdots, z_n)$ being local coordinates of $F$. An upper index $i=0$ 
	does not contribute any singular factor of $x$ due to the structure of the inverse 
	$g^{-1}$. A lower index $i=0$ indicates a differentiation by $\partial_x \in x^{-1}\V_b$ Hence an index $i=0$ (as a combination of a lower and upper index) contributes 
	$x^{-1}\V_b$ up to a term of type $A$. \medskip
	
	Similarly, an upper index $i>0$ 
	contributes a singular factor $x^{-1}$ due to the structure of the inverse 
	$g^{-1}$. A lower index $i>0$ indicates a differentiation by $\partial_{z_j} \in \V_b$ Hence an index $i>0$ (as a combination of a lower and upper index) contributes
	$x^{-1}\V_b$ up to a term of type $A$. Hence in total we find
	\begin{align} \label{L term 1}
		-g^{ij} \p_i \p_j c_{kl} & = -a^{-1} \wt{g}^{ij} \p_i \p_j c_{kl} 
		+ A \cdot (x^{-1}\V_b)^2 c_{kl},
	\end{align}
	for some $A \in \mathcal{C}^{k,\a}_{\textup{ie}}(M \times [0,T],S)_{\ov{\g}}$.
	In order to study the remaining terms in \eqref{Laplacian}, involving Christoffel
	symbols, we note that the Christoffel symbols $\G_{ij}^k$ with respect to the metric $g(t)$ are related to the Christoffel symbols $\wt{\G}_{ij}^k$ of the initial metric $\wt{g}$ as follows
	\begin{equation} \label{Christoffel} \begin{split}
	\G_{ij}^k & = \wt{\G}_{ij}^k + \frac{1}{2}a^{-1}\wt{g}^{km} (\p_i a \, \wt{g}_{jm} 
	+ \p_j a \, \wt{g}_{im} - \p_m a \, \wt{g}_{ij} + \p_i \w_{jm} \\
	& + \p_j \w_{im} - \p_m \w_{ij}) +  \frac{1}{2} (- a^{-2} \wt{g}^{kl} \wt{g}^{mp} \w_{pl} + a^{-2} g^{ml} \wt{g}^{kr}\wt{g}^{pq} \w_{lp} \w_{rq})  \\
	& \times (a (\p_i \wt{g}_{jm} + \p_j \wt{g}_{im} - \p_m \wt{g}_{ij}) + \p_i a \, \wt{g}_{jm} + \p_j a \, \wt{g}_{im} - \p_m a \, \wt{g}_{ij} \\
	& + \p_i \w_{jm} + \p_j \w_{im} - \p_m \w_{ij}) 
	\end{split}\end{equation}
	This allows us to conclude by counting upper and lower indices as above that for any 
	$\{Y_\ell\}_{\ell}$ with $Y_0 := \partial_x$ and $Y_i := x^{-1}\partial_{z_i}$ for $i=1,\cdots, n$, we have
	\begin{align}
		(\D c) (Y_k,Y_l)  = a^{-1} (\wt{\D} c) (Y_k,Y_l) 
		+ \left(A \cdot (x^{-1}\V_b)^2 + B \cdot x^{-1}\V_b + C \right) c(Y_k,Y_l),
	\end{align}
	for some terms $A,B,C$ as in the statement of the theorem. Computing further
	\begin{equation}\begin{split}
		&\D  = a^{-1/2} \wt{\D} \, a^{-1/2} +
		+ \left(A \cdot (x^{-1}\V_b)^2 + B' \cdot x^{-1}\V_b 
		+ C' \right), \\
		&\textup{where} \ B'=B-a^{-1/2} (\wt{\nabla} a^{-1/2}), 
		\quad C'=C- a^{-1/2} (\wt{\D} a^{-1/2}).
	\end{split}\end{equation}
	Note that the terms $B',C'$ still have the same H\"older regularity as $B,C$.
	Now we turn to the Lichnerowicz Laplacian. Recall that it is given by 
	\[ \D_L c_{kl} = \D c_{kl} - 2 g^{pq} R^r_{qkl} c_{rp} + g^{pq} R_{kp} c_{ql} + g^{pq}R_{lp} c_{kq}.\]
	To deal with the extra terms involving the Riemann curvature tensor and the Ricci curvature on the right-hand side, we note the following formula for the components of the curvature tensor in local coordinates
	\begin{equation}
		R^l_{ijk} = \p_i \G_{jk}^l - \p_j \G_{ik}^l + \G^p_{jk} \G^l_{ip} - \G^p_{ik} \G^l_{jp},
	\end{equation}  
	which also leads to a formula for the components of the Ricci tensor
	\begin{equation}
		R_{jk} = R^i_{ijk} = \p_i \G_{jk}^i - \p_j \G_{ik}^i + \G^p_{jk} \G^i_{ip} - \G^p_{ik} \G^i_{jp}.
	\end{equation}
	These extra terms can be treated using \eqref{Christoffel}, which leads to the claimed expansion for the Lichnerowicz Laplacian.
\end{proof}

\begin{cor}\label{XR}
Under the conditions of Theorem \ref{L expansion} we conclude
	\begin{equation}
		(\p_t +a^{-1/2} \wt{\D} \, a^{-1/2})\ric =: X(\ric) \in 
		\mathcal{C}^{k-1,\a}_{\textup{ie}}(M \times [0,T],S)_{-4+2\ov{\g}}.
	\end{equation}
\end{cor}

\begin{proof}Recall the evolution \eqref{evol equ Ricci} 
	of Ricci curvature along the Ricci de Turck flow
	\begin{equation}
		(\p_t + \D_L) \ric_{jk} = \n_m \ric_{jk} W^m + 
		\ric_{jm} \n_k W^m + \ric_{km} \n_j W^m,
	\end{equation}
	By \eqref{W-regularity} and \eqref{Ricci-regularity} we conclude
	$$
	(\p_t + \D_L) \ric \in 
	\mathcal{C}^{k-1,\a}_{\textup{ie}}(M \times [0,T],S)_{-4+2\ov{\g}}.
	$$
	The statement now follows from Theorem \ref{L expansion} and \eqref{Ricci-regularity}.
\end{proof}

%%%%%%%%%%%%%%%%%%%%%
\subsection{Mapping properties of the heat parametrix for $L:=a^{1/2} \wt{\D} \, a^{-1/2}$} 
%%%%%%%%%%%%%%%%%%%%%
We first point out the following relation between $L=a^{1/2} \wt{\D} \, a^{-1/2}$
and the Lichnerowicz Laplace operator $\wt{\D_L}$
\begin{equation}\label{DD}\begin{split}
L=a^{1/2} \wt{\D} \, a^{-1/2} &= \wt{\D} + a^{1/2} (\wt{\D} \, a^{-1/2})  
+ a^{1/2} (\wt{\nabla} \, a^{-1/2}) \wt{\nabla} \\ 
&= \wt{\D}_L + \left(B \cdot \wt{\nabla} + C \right) =: \wt{\D}_L + P,
\end{split}\end{equation}
where $B$ and $C$ are terms with same regularity as in Theorem \ref{L expansion}.
The heat operator for $\wt{\D_L}$ has been studied in \cite[Theorem 3.1]{Ver-Ricci}, which 
asserts that for $n\geq 3$, any $\g \in (0,2+\min \, \{\mu_0,\mu_1\})$ and $\alpha \in (0,1)$ sufficiently small, the following is a bounded mapping (we simplify notation by writing 
$\cH^{k, \A}_{\gamma} \equiv \cH^{k, \A}_{\gamma, \g}$)
\begin{align}\label{original-mapping} 
e^{-t\wt{\D_L}}:  \cH^{k, \A}_{-4+\gamma} (M\times [0,T], S)
\to \cH^{k+2, \A}_{-2+\gamma} (M\times [0,T], S).
\end{align}
Since by \eqref{DD}, $L$ and $\wt{\D_L}$
differ only by lower order terms $P$, the heat operator construction in 
\cite{Ver-Ricci} carries over to get a heat parametrix for $L$
as well. Thus, exactly as in \eqref{original-mapping}, we find
\begin{align}\label{original-mapping2} 
H := e^{-t L}:  \cH^{k, \A}_{-4+\gamma} (M\times [0,T], S)
\to \cH^{k+2, \A}_{-2+\gamma} (M\times [0,T], S).
\end{align}
The following mapping properties are needed later in this section,
and follow by similar arguments as in \cite[Theorem 3.1]{Ver-Ricci} and 
\eqref{original-mapping}. 

\begin{lem}\label{Lemma 4.2}
	Consider smooth functions $\rho \in C^\infty(\overline{M})$ smooth up to the boundary, 
	$\eta \in C^\infty_0(M)$ smooth with compact support away from the conical singularity.
    Let $\nu \in C^\infty_0(M,TM)$ be a smooth $1$-form and 
    $\w \in C^\infty_0(M,{}^{ib}T^* M \otimes {}^{ib}T^* M)$ a smooth $(0,2)$-tensor, both with 
    compact support away from the conical singularity. Let $H$ be the heat parametrix of 
    $L=a^{1/2} \wt{\D} \, a^{-1/2}$. Then 
	\begin{align*} 
	& E_0 := \eta \, H \, \rho: 
	 \cH^{k, \A}_{-4+\gamma} (M \times [0,T],S)
	\to  \cH^{k+2, \A}_{-2+\gamma} (M \times [0,T],S), \\
	& \wt{E}_0 := \nu \otimes H \, \rho: 
     \cH^{k, \A}_{-4+\gamma} (M \times [0,T],S)
	\to  \cH^{k+2, \A}_{-2+\gamma} (M \times [0,T],
	{}^{ib}T^* M \otimes S),\\
	& \wt{\wt{E}}_0 := \w \otimes H \, \rho : 
	 \cH^{k, \A}_{-4+\gamma} (M \times [0,T],S)
	\to  \cH^{k+2, \A}_{-2+\gamma} (M \times [0,T],
	{}^{ib}T^* M^{\otimes^2} \otimes S), 
	\end{align*}
	with $\|E_0\| \to 0$, $\|\wt{E}_0\| \to 0$ and $\|\wt{\wt{E}}_0\| \to 0$ as $T \to 0$. 
\end{lem}

The next result is a straightforward consequence of Lemma \ref{Lemma 4.2}.

\begin{cor}\label{cor to Lemma 4.2}
In the notation of Lemma \ref{Lemma 4.2} we find
	\begin{align*} 
	& F_0 := \nu \otimes \n \, H \, \rho : 
	 \cH^{k, \A}_{-4+\gamma} (M \times [0,T],S)
	\to  \cH^{k+1, \A}_{-3+\gamma} (M \times [0,T],
	{}^{ib}T^* M^{\otimes^2} \otimes S), \\
	& G_0 := g^{ij} \nu_i \n_j H \, \rho :  \cH^{k, \A}_{-4+\gamma} (M \times [0,T],S)
		 \to  \cH^{k+1, \A}_{-4+\gamma} (M \times [0,T], S),
	\end{align*}
	with $\|F_0\| \to 0$, $\|G_0\| \to 0$ as $T \to 0$.
\end{cor}

\begin{proof}
	By Lemma \ref{Lemma 4.2} we have 
	\[ \n (\nu \otimes H \rho ): 	
	 \cH^{k, \A}_{-4+\gamma} (M \times [0,T],S)
	\to  \cH^{k+1, \A}_{-3+\gamma} (M \times [0,T],
	{}^{ib}T^* M^{\otimes^2} \otimes S), \]
	with $\|\n (\nu \otimes H\rho )\| \to 0$ as $T \to 0$. 
	On the other hand
	\[ \n (\nu \otimes H \rho ) = \n \nu \otimes H\rho  + \nu \otimes \n H\rho. \]
	Hence the mapping properties of $F_0$ follow by applying 
	Lemma \ref{Lemma 4.2} to $\n \nu \otimes H\rho$. 
	The mapping properties of $G_0$ follow from the ones of $F_0$, 
	as $G_0$ is simply $F_0$ composed with a contraction.
\end{proof}

%%%%%%%%%%%%%%%%%%%%%
\subsection{Heat operator parametrix for $a^{-1/2} \wt{\D} \, a^{-1/2}$} 
%%%%%%%%%%%%%%%%%%%%%

Our goal in this section is the existence of inverses $\mathcal{Q}$ and $\mathcal{R}$ for the parabolic operator 
$P := \p_t + a^{-1}L$, where $a$ is the pure trace part of
$g(t)$ with respect to the initial metric $g$ and is positive uniformly bounded away from zero for short time. These parametrices are constructed out of the heat parametrix $H$ for 
$L=a^{1/2} \wt{\D} \, a^{-1/2}$, using \eqref{original-mapping2}, 
Lemma \ref{Lemma 4.2} and Corollary \ref{cor to Lemma 4.2}.
Our main result in this subsection is as follows.
\begin{thm}\label{Q neu}
Let $n\geq 3$. Consider any $\g \in (0,2+\min \, \{\mu_0,\mu_1\})$ and $\alpha \in (0,1)$ sufficiently small, such that \eqref{original-mapping} holds. Then there exists $T_0 > 0$ 
sufficiently small and a bounded linear map
\[ \mathcal{Q}:  \cH^{k, \A}_{-4+\gamma} (M \times [0,T_0],S)
\to  \cH^{k+2, \A}_{-2+\gamma} (M \times [0,T_0],S), \]
such that if $f \in  \cH^{k, \A}_{-4+\gamma} (M \times [0,T_0],S)$, 
then $u = \mathcal{Q}f$ solves the initial value problem
\[ (\p_t + a^{-1}L )u = f, \qquad u(\cdot,0) = 0. \]
\end{thm}

The proof of that theorem will occupy the rest of this subsection. 
Before we proceed, let us note an immediate consequence.

\begin{thm}\label{R neu}
Let $n\geq 3$. Consider any $\g \in (0,2+\min \, \{\mu_0,\mu_1\})$ and $\alpha \in (0,1)$ sufficiently small, such that \eqref{original-mapping} holds. Then there exists $T_0 > 0$ 
sufficiently small and a bounded linear map
\[ \mathcal{R}:  \cH^{k+2, \A}_{-2+\gamma} (M,S) 
\to  \cH^{k+2, \A}_{-2+\gamma} (M \times [0,T_0],S), \]
such that if $u_0 \in  \cH^{k+2, \A}_{-2+\gamma} (M,S)$, 
then $u = \mathcal{R}u_0$ solves the initial value problem
\[ (\p_t + a^{-1}L )u = 0, \qquad u(\cdot,0) = u_0. \]
\end{thm}

\begin{proof}
We have $L u_0 \in  \cH^{k, \A}_{-4+\gamma} (M,S)$ and thus $a^{-1}Lu_0 \in  
\cH^{k, \A}_{-4+\gamma} (M \times [0,T_0],S)$. We set $\mathcal{R}u_0 := u_0 + \mathcal{Q}(a^{-1}Lu_0)$, which yields the 
desired solution to the initial value problem above.
\end{proof}

%%%%%%%%%%%%%%%%%%%%%%%%%%%%%%%%%%%%%%%%%%%%%%%%%
\subsubsection{Construction of a boundary parametrix} \ \medskip
%%%%%%%%%%%%%%%%%%%%%%%%%%%%%%%%%%%%%%%%%%%%%%%%%

We now begin with the proof of Theorem \ref{Q neu}. The proof 
idea is to construct boundary and interior parametrices, and glue 
them together to an approximate solution to $(\p_t + a^{-1}L )$.
Provided the error is sufficiently "nice", we can obtain a solution $\mathcal{Q}$
by a Neumann series argument. We follow the analytic path outlined in the
work by the third named author jointly with Bahuaud \cite{BahVer}.
\medskip

We start by introducing the notion of a \textit{reference covering} which is a special case of a covering considered in \cite{BahVer}. Let $U = [0,1)_x \times (-1,1)^n_z$
be a model half-cube. For each $p \in \p M$, there is a coordinate chart 
$\Phi_p: U \to W_p$ onto an open neighborhood $W_p \subset \cC(F) \subset \ov{M}$,
centered around $p_i$. Due to compactness of $\p M$, there are finitely many such charts $\{W_i,  \Phi_{p_i}\}_{i=1}^N$ covering $\cC(F)$. Together with 
the open set $W_0 = \ov{M} \backslash 
([0,\frac{1}{2})\times F)$ we obtain an open cover of $\ov{M}$ that we call 
a reference covering.\medskip

Let $\sigma: \R^+ \to [0,1]$ be a smooth cutoff function with $\sigma(s) = 1$ for $s \le 1/4$ and $\sigma(s) = 0$ for $s \ge 1/2$. Denote by $(x,z) \in \Phi_{p_i}^{-1}(W_i)$ the local coordinates on the coordinate chart $W_i$, centered around $p_i \in \p M$. Then we define 
$$
\varphi_i \left(\Phi_{p_i}(x,z)\right) := \sigma(x)\s(\|z\|), \quad 
\psi_i \left(\Phi_{p_i}(x,z)\right) := \sigma\left(\frac{x}{2}\right)
\s\left(\frac{\|z\|}{2}\right).
$$ 
Notice that $\varphi_i, \psi_i \in C^\infty(\wt{M})$ are both identically 
$1$ near $(0,0)$, and $\psi_i \equiv 1$ on $\supp \varphi_i$. 
Furthermore, $\varphi_i, \psi_i \in \textup{C}^{\alpha}_{\textup{ie}}(M)$, since they are constant 
near the cone points. These functions are illustrated in the radial direction 
in Figure \ref{fig:CutOff}.
\begin{figure}[h]
	\begin{center}
		
		\begin{tikzpicture}[scale=1.3]
		%coordinate axes
		\draw[->] (-0.2,0) -- (9,0);
		\draw[->] (0,-0.2) -- (0,2.2);
		
		% 1 tick on vertical axis
		\draw (-0.2,2) node[anchor=east] {$1$};
		%delta, delta' ticks
		
		%\draw (9.3,-0.2) node[anchor=north] {$1$} -- (9.3,0.2);
		\draw (0,-0.45) node {$0$};
		\draw (0,2) -- (4,2);
		%phi 
		\draw (2,2) .. controls (3.4,2) and (2.6,0) .. (4,0);
		\draw[dashed] (2,-0.2) -- (2,2.2);
		\draw[dashed] (4,-0.2) -- (4,2.2);
		%psi
		\draw (4,2) .. controls (5.4,2) and (4.6,0) .. (8,0);
		\draw[dashed] (4,-0.2) -- (4,2.2);
		\draw[dashed] (8,-0.2) -- (8,2.2);
		
		\draw (2.5,1) node {$\varphi_i$};
		\draw (4.5,1) node {$\psi_i$};
		
		\draw (2,-0.5) node {$1/4$};
		\draw (4,-0.5) node {$1/2$};
		\draw (8,-0.5) node {$1$};

		\end{tikzpicture}
		
		\caption{The cutoff functions $\varphi_i, \psi_i$.}
		\label{fig:CutOff}
	\end{center}
\end{figure}
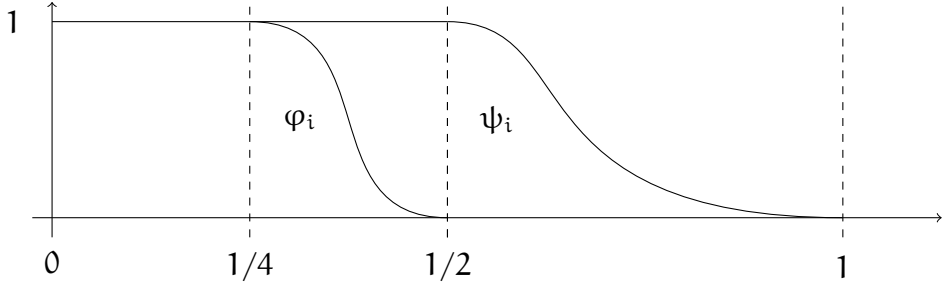

Next, for $\ve \in (0,1)$ to be specified later, and 
any $p = z_0 \in W_i \cap \p M$ we set
% FOR TOBIAS: We do not rescale in zhe z direction. This leaves the Hoelder differences
% in the z direction bounded, but Hoelder differences in x direction become singular 
% in epsilon. So what?! Jointly, the Hoelder differences are still singular in epsilon!
% without any z rescaling! Also, z rescaling is not allowed, since then we dont get 
% a finite cover with the number of covering functions independent of epsilon!
\[ \wt{\varphi}_{i,p}(x,z) := 
\varphi_i \left(\frac{x}{\ve}, z-z_0 \right), 
\;\; \wt{\psi}_{i,p}(x,z) := \psi_i \left(\frac{x}{\ve}, z-z_0 \right). \]
This defines smooth cutoff functions 
$\wt{\varphi}_{i,p}, \wt{\psi}_{i,p} \in C^\infty(\ov{M})$
which again lie in $\textup{C}^{k,\alpha}_{\textup{ie}}(M)$, with the H\"older norms
bounded by $\ve^{-\alpha} \| \varphi_i \|_{\textup{C}^{k,\alpha}_{\textup{ie}}(M)}$ and 
$\ve^{-\alpha} \| \psi_i \|_{\textup{C}^{k,\alpha}_{\textup{ie}}(M)}$,
respectively
\begin{equation}\label{hoelder-phi-estimate-0}
\| \wt{\varphi}_{i,p} \|_{\textup{C}^{k,\alpha}_{\textup{ie}}(M)} \leq \ve^{-\alpha} \| \varphi_i \|_{\textup{C}^{k,\alpha}_{\textup{ie}}(M)}, \quad 
\| \wt{\psi}_{i,p} \|_{\textup{C}^{k,\alpha}_{\textup{ie}}(M)} \leq \ve^{-\alpha} \| \psi_i \|_{\textup{C}^{k,\alpha}_{\textup{ie}}(M)}.
\end{equation}
Let $L_{i}$ be the lattice of points in the coordinate chart $W_i$ of the form $\{ (0, w) \, | \, w \in \Z^n\}$. 
By construction, every point on $\p M$ lies 
in the support of at most a fixed number (independent of $\ve$) of functions $\{ \wt{\psi}_{i,p} \, | \, i = 1,\dots,N, 
p \in L_{i}\}$. Furthermore, let $F_{\ve} = F_{\ve}(x)$ be a cutoff function with $\supp F_{\ve} \subset \{x \ge \ve/4\} \subset M$,
which is identically $1$ on the set $\{x \ge \ve\}$. Let $G_{\ve}(x) = F_{\ve}(\frac{x}{2})$.
For any $p \in L_i$ we normalize
\begin{equation}
\varphi_{i,p} := \frac{\wt{\varphi}_{i,p}}{G_{\ve} + \sum_k \sum_{q \in L_{k}} \wt{\varphi}_{k,q}}, \;\; \psi_{i,p} := 
\frac{\wt{\psi}_{i,p}}{F_{\ve} + \sum_k \sum_{q \in L_{k}} \wt{\varphi}_{k,q}}.
\end{equation}
By \eqref{hoelder-phi-estimate-0}, these functions 
again lie in $\textup{C}^{\alpha}_{\textup{ie}}(M)$, with 
\begin{equation}\label{hoelder-phi-estimate}
\begin{split}
&\| \varphi_{i,p} \|_{\textup{C}^{k,\alpha}_{\textup{ie}}(M)} \leq \textup{const} 
\cdot \ve^{-\alpha} \| \varphi_i \|_{\textup{C}^{k,\alpha}_{\textup{ie}}(M)}
\leq \textup{const} \cdot \ve^{-\alpha}, \\ 
&\| \psi_{i,p} \|_{\textup{C}^{k,\alpha}_{\textup{ie}}(M)} \leq \textup{const} 
\cdot \ve^{-\alpha} \| \psi_i \|_{\textup{C}^{k,\alpha}_{\textup{ie}}(M)}
\leq \textup{const} \cdot \ve^{-\alpha}.
\end{split}
\end{equation}
Finally note that due to normalization, for any $0 < \ve < 1$, the functions
\[ \Phi := \sum_i \sum_{p \in L_{i}} \varphi_{i,p}, \;\; \Psi := \sum_i \sum_{p \in L_{i}} \psi_{i,p} \]
are smooth cutoff functions in $C^\infty(\ov{M})$ which are identically $1$ in a neighborhood of the cone singularity. We can now introduce 
an approximate boundary parametrix. Consider any 
$f \in  \cH^{k, \A}_{-4+\gamma} (M \times [0,T_0],S)$
and denote by $H_p$ the heat kernel of\footnote{Recall that 
$g(p,t) = a(p,t)g(p) \oplus \w(p,t)$.} $a(p,0)^{-1}L$ for any $p\in \p M$
(note this is simply a rescaling of $L$ by a positive constant). Then we define 
our boundary parametrix by
\begin{align}\label{Q-parametrix}
Q_b f := \sum_{j=1}^N \sum_{p \in L_{j}} \psi_{j,p}H_p [\varphi_{j,p}f]. 
\end{align}

\begin{lem}\label{Ejp}
The solution $u_p := \psi_{j,p} H_p [\varphi_{j,p}f]$ satisfies
\[ (\p_t + a^{-1}L) u_p = \varphi_{j,p} f + E_{j,p}^0 f + E_{j,p}^1 f, \]
where the operators $E^0_{j,p}$ and $E^1_{j,p}$ are as follows.
\begin{itemize}
  \item[(i)] $E^0_{j,p}$ is a bounded linear map on $\cH^{k, \A}_{-4+\gamma} (M \times [0,T],S)$, 
  and there exists a constant $C > 0$ independent of $j,p$, such that for $T/\ve^2< 1$
  $$\|E^0_{j,p}f\| \le C(\ve^\g + T^{\a/2}) \|\varphi_{j,p}f\|,$$
  \item[(ii)] $E^1_{j,p}$ is a bounded linear map on $\cH^{k, \A}_{-4+\gamma} (M \times [0,T],S)$, 
  with operator norm satisfying $\|E^1_{j,p}\| \to 0$ as $T \to 0^+$.
\end{itemize}
\end{lem}

\begin{proof}
We simplify notation by omitting the subscripts on $\varphi, \psi$ and $E^0, E^1$. 
Then we compute in view of \eqref{DD}
\begin{align*}
&(\p_t + a^{-1}L )u_p 
 = (\p_t + a^{-1}L )(\psi H_p [\varphi f])\\
& = \psi \p_t H_p [\varphi f] +a^{-1}\left(\wt{\D} \psi \cdot H_p [\varphi f] + 
2 \wt{g}^{ij} \wt{\n}_i \psi \left( \wt{\n}_j H_p [\varphi f] + B \cdot H_p [\varphi f]\right) + \psi L H_p [\varphi f]\right) \\
& = \psi (\p_t + a^{-1}(p,0)L) H_p [\varphi f] + \psi (a^{-1}-a^{-1}(p,0)) L H_p [\varphi f] \\
&+a^{-1} \wt{\D} \psi \cdot H_p [\varphi f] +a^{-1} 2 \wt{g}^{ij} \wt{\n}_i \psi 
\left( \wt{\n}_j H_p [\varphi f] + B \cdot H_p [\varphi f]\right) =: \psi \varphi f + E^0 f + E^1 f,
\end{align*}
where the operators $E^0$ and $E^1$ are explicitly given by
\begin{align*}
& E^0 f := \psi (a^{-1}-a^{-1}(p,0)) L H_p [\varphi f],   \\
& E^1 f := a^{-1} \wt{\D} \psi \cdot H_p [\varphi f] + 2a^{-1} \wt{g}^{ij} \wt{\n}_i \psi 
\left( \wt{\n}_j H_p [\varphi f] + B \cdot H_p [\varphi f]\right).
\end{align*}
Since $\psi \equiv 1$ on $\supp \varphi$, any derivative of $\psi$ vanishes on a neighborhood of the 
conical singularity. Hence the claimed mapping properties of $E^1$ follow from Lemma \ref{Lemma 4.2} and Corollary \ref{cor to Lemma 4.2}.
What is left is establishing the mapping properties of $E^0$. 
Writing out definitions of the various H\"older norms, we find
\begin{equation} \label{Schauder-E0}
\begin{split}
\| E^0 f \|_{\cH^{k, \A}_{-4+\gamma}} &\le C \, \|\psi \, (a^{-1}-a^{-1}(p,0))\|_{\mathcal{C}^{k,\alpha}_{\textup{ie}}} 
\cdot \|L H_p [\varphi f]\|_{\cH^{k, \A}_{-4+\gamma}}.
\end{split}
\end{equation}
We write $p = (0,z_0)$. For any $(x,z)$ in the support of $\psi$ we obtain
\begin{align*}
|a^{-1}(x,z,t) - a^{-1}(0,z_0,0)| & \le 
|x|^\g \|x^{-\g} (x \p_x) a^{-1}\|_\infty  + t^{\frac{\a}{2}} \|a^{-1}\|_\a  \\
& \le |x|^\g \|x^{-\g} (x \p_x) a\|_\infty  + t^{\frac{\a}{2}} \|a\|_\a  \\
& \le C (d_M(p,q)^\g + t^{\frac{\a}{2}}) \|a\|_{C^{1,\alpha}_{\textup{ie}}(M\times [0,T],S_1)^b_\g} \\
& \le C (\varepsilon^\g + T^{\frac{\a}{2}}) \|a\|_{C^{1,\alpha}_{\textup{ie}}(M\times [0,T],S_1)^b_\g},
\end{align*}
where we used that $a(0,\cdot,t)$ is constant in $z$, 
and $\p_x a^{-1} = -\p_x a \cdot a^{-2}$. Recall
the notation $\mathscr{D}$ of Definition \ref{funny-spaces0}.
Then for any partial differential operator $X \in \mathscr{D}$
of order at least $1$, we compute similar to above
for any $(x,z) \in \supp \psi$
\begin{align*}
&|X (a^{-1}(x,z,t) - a^{-1}(0,z_0,0))| = |X a^{-1}(x,z,t)|  \le 
|x|^\g \|x^{-\g} X a^{-1}\|_\infty \\ & \le C \varepsilon^\g \|a\|_{C^{1,\alpha}_{\textup{ie}}(M\times [0,T],S_1)^b_\g} 
 \le C (\varepsilon^\g + T^{\frac{\a}{2}})  \|a\|_{C^{1,\alpha}_{\textup{ie}}(M\times [0,T],S_1)^b_\g}.
\end{align*}
Analogous estimates hold for H\"older differences. Noting 
\eqref{hoelder-phi-estimate}, we conclude
\begin{align}\label{A-est}
 \|\psi \, (a^{-1}-a^{-1}(p,0))\|_{\mathcal{C}^{k,\alpha}_{\textup{ie}}} 
 \le C \varepsilon^{-\alpha} (\varepsilon^\g + T^{\frac{\a}{2}}) \|a\|_{\mathcal{C}^{k+1,\alpha}_{\textup{ie}}(M\times [0,T],S_1)^b_\g}.
\end{align}
Now we estimate $L H_p [\varphi f]$. We first note that it vanishes identically at $t=0$. 
This is obtained from the following mapping properties of the heat operator
\begin{align*}
&H_p \circ \varphi : \cH^{k, \A}_{-4+\gamma} (M\times [0,T], S)
\to \cH^{k+2, \A}_{-2+\gamma} (M\times [0,T], S), \\
&H_p \circ \varphi : \cH^{k, \A}_{-4+\gamma} (M\times [0,T], S)
\to t^{\frac{\alpha}{2}}\cH^{k+2, 0}_{-2+\gamma} (M\times [0,T], S),
\end{align*}
where the first mapping property is due to \eqref{original-mapping}
and the second is obtained by similar arguments, converting the
lesser target regularity into a time weight. Hence 
$L H_p [\varphi f] \in \cH^{k, 0}_{-4+\gamma}$ 
vanishes identically at $t=0$ and hence 
\begin{align}
&\|L H_p [\varphi f]\|_{\cH^{k, 0}_{-4+\gamma}} 
= \|L H_p [\varphi f] - L H_p [\varphi f](t=0)\|_{\cH^{k, 0}_{-4+\gamma}}
\\ &\le C T^{\frac{\alpha}{2}} \|[\varphi f]\|_{\cH^{k, \alpha}_{-4+\gamma}}
\le C  (\varepsilon^\alpha + T^{\frac{\a}{2}}) \|[\varphi f]\|_{\cH^{k, \alpha}_{-4+\gamma}}.
\end{align}
Analogous estimates hold for H\"older differences, if we take supremums over
$\supp \psi$. Hence overall we arrive at the estimate 
\begin{align}\label{B-est}
\|L H_p [\varphi f]\|_{\cH^{k, \alpha}_{-4+\gamma}(\supp \psi \times [0,T], S)} 
\le C  (\varepsilon^\alpha + T^{\frac{\a}{2}}) \|[\varphi f]\|_{\cH^{k, \alpha}_{-4+\gamma}}.
\end{align}
Plugging the estimates \eqref{A-est} and \eqref{B-est} into \eqref{Schauder-E0}, 
we conclude for $T/\ve^2< 1$
\begin{align*}
\| E^0 f \|_{\cH^{k, \A}_{-4+\gamma}} & \le 
C  (\varepsilon^\g + T^{\frac{\a}{2}}) \ve^{-\alpha}  (\varepsilon^\alpha + T^{\frac{\a}{2}}) 
 \|[\varphi f]\|_{\cH^{k, \alpha}_{-4+\gamma}} \\
 & \le 
C  (\varepsilon^\g + T^{\frac{\a}{2}})  (1 + (T/\ve^2)^{\frac{\a}{2}}) 
 \|[\varphi f]\|_{\cH^{k, \alpha}_{-4+\gamma}} \\
 & \le 
C  (\varepsilon^\g + T^{\frac{\a}{2}}) 
 \|[\varphi f]\|_{\cH^{k, \alpha}_{-4+\gamma}}.
 \end{align*}
This finishes the proof. 
\end{proof}

\begin{prop}\label{Qb}
Let $\a \in (0,\min \, \{\gamma, 1\})$. For every $\delta > 0$, there exists $\ve > 0$ and $T_0 > 0$ sufficiently small, 
such that the heat parametrix defined in \eqref{Q-parametrix}
\[ Q_b: \cH^{k, \A}_{-4+\gamma} (M \times [0,T_0],S) \to \cH^{k+2, \A}_{-2+\gamma} (M \times [0,T_0],S), \]
is a bounded linear map, solving 
\[ (\p_t + a^{-1}L )(Q_b f) = \Phi f + E^0 f + E^1 f, \]
where $E^0$ and $E^1$ are bounded linear maps on $\cH^{k, \A}_{-4+\gamma} (M \times [0,T_0],S)$
with operator norms $\|E^0\| < \delta$ and $\|E^1\| \to 0$ as $T_0 \to 0^+$. 
\end{prop}

\begin{proof}
By Lemma \ref{Ejp} we compute for any $f\in \cH^{k, \A}_{-4+\gamma} (M \times [0,T_0],S)$
\begin{align*}
(\p_t + a^{-1}L) (Q_b f) & = \sum_{j=1}^N \sum_{p \in L_j} (\p_t + a^{-1}L) \psi_{j,p}H_p [\varphi_{j,p}f] 
= \Phi f + E^0 f + E^1 f, 
\end{align*}
where we have defined for $i=0,1$
\[ E^i f := \sum_{j=1}^N \sum_{p \in L_j} E^i_{j,p} f. \]
The operators $E^i_{j,p}$, and hence also both $E^0$ and $E^1$, are bounded operators on $\cH^{k, \A}_{-4+\gamma} (M \times [0,T_0],S)$
by Lemma \ref{Ejp}. It remains to estimate their operator norms as $T_0\to 0$. The fact that $\|E^1\| \to 0$ as $T_0 \to 0^+$ is a direct consequence of
the second statement in Lemma \ref{Ejp}. For the estimate of $\|E^0\|$ we argue as follows. We choose $\ve > 0$ and $T_0 > 0$ with $T_0 < \ve^2$. 
Then by Lemma \ref{Ejp} and \eqref{hoelder-phi-estimate} we have
\begin{align*}
\|E^0_{j,p}f\|_{\cH^{k, \A}_{-4+\gamma}} 
&\le C(\ve^\g + T_0^{\a/2}) \|\varphi_{j,p}f\|_{\cH^{k, \A}_{-4+\gamma}}
\\ &\le C \ve^{-\alpha}(\ve^\g + T_0^{\a/2}) \| f\|_{\cH^{k, \A}_{-4+\gamma}}
\\ &= C (\ve^{\g-\alpha} + \ve^{-\alpha} T_0^{\a/2}) \| f\|_{\cH^{k, \A}_{-4+\gamma}}
\end{align*}
Now set $\d' := \frac{\d}{ \sum_j |L_{j}|N}> 0$ for a given $\d>0$. Then fix $\ve > 0$ sufficiently small $C\ve^{\g-\alpha} < \d'/2$
(note that $\g - \alpha > 0$). For the fixed $\ve>0$ choose $T_0 > 0$ sufficiently small, such that $C\ve^{-\alpha} T_0^{\a/2}< \d'/2$. 
These choices yield $\|E^0\| < \d$,
and the proof is finished. 
\end{proof}

%%%%%%%%%%%%%%%%%%%%%%%
\subsubsection{Construction of an interior parametrix} 
%%%%%%%%%%%%%%%%%%%%%%%%

Next we construct an approximate interior parametrix. This construction is analogous to the one in \cite{BahVer}, 
but for the convenience of the reader we repeat it here. \medskip

Recall that the radial function of the cone $x: \cC(F) \to (0,1)$
is extended smoothly to a nowhere vanishing function $x\in C^\infty(M)$. We assume that $x \ge 1$ outside of the
singular neighborhood $\cC(F)$. For $\ve > 0$ small enough, $Y_{\ve} := \{x \ge \ve/2\} \subset M$ is a manifold with 
smooth boundary $\{\ve/2\}\times F$. Let $\ov{Y}$ denote the double of $Y_{\ve}$. Since the Riemannian metric on 
$\ov{Y}$ need not be smooth, we smoothen the metric in a small collar neighborhood of $\{\ve/2\}\times F$, such that the 
metrics on $\ov{Y}$ and $M$ coincide over $Y_{2\ve}$. \medskip

Since $\Phi \in C^\infty(\ov{M})$ is by construction identically $1$ in $\{x < \ve\}$, the function $1-\Phi$ defines a 
smooth cutoff function on the closed double $\ov{Y}$, which is again denoted by $1-\Phi$. 
Let $\ov{P}$ denote the extension of $P = \p_t + a^{-1}L $ to a uniformly parabolic operator on $\ov{Y}$. 
Note that $\cH^{k, \A}_{-4+\gamma} \restriction Y_{\ve} \equiv \mathcal{C}^{k,\A}_{\textup{ie}}(Y_{\ve}\times [0,T_0],S)$. 
Let $\wt{Q}_i: \mathcal{C}^{k,\A}_{\textup{ie}}(Y_{\ve}\times [0,T_0],S) \to \mathcal{C}^{k+2,\A}_{\textup{ie}}(Y_{\ve}\times [0,T_0],S)$ 
be the solution operator of the inhomogeneous Cauchy problem
\[ \ov{P}u = (1-\Phi)f, \qquad u(\cdot,0) = 0. \]
Also let $\chi$ be any smooth cutoff function which is identically $1$ on $\supp (1-\Phi)$. Then we define the interior parametrix 
for any $f\in \cH^{k, \A}_{-4+\gamma}$ by
\[ Q_i f = \chi \wt{Q}_i [(1-\Phi)f]. \]

%%%%%%%%%%%%%%%%%%%%%%%
\subsubsection{Construction of the parametrix}
%%%%%%%%%%%%%%%%%%%%%%
From the approximate boundary and interior parametrices we obtain an approximate parametrix by 
\[ Qf := Q_b f + Q_i f \]
for $f \in \cH^{k, \A}_{-4+\gamma}(M \times [0,T_0],S)$. From there we get a parametrix as follows. 
\begin{proof}[Proof of Proposition \ref{Q neu}]
By Proposition \ref{Qb} we conclude
\begin{align*}
(\p_t + a^{-1}L) Qf  = \Phi f + E^0 f + E^1 f + (1-\Phi)f + E^2 f=: f + Ef,  
\end{align*} 
where $E^2 f := [\chi, a^{-1}L] \wt{Q}_i [(1-\Phi)f]$. Now similarly as in
Lemma \ref{Lemma 4.2} and Corollary \ref{cor to Lemma 4.2} we find that 
$\|E^2\| \to 0$ as $T_0 \to 0$. Hence we can choose $T_0$ sufficiently small, such that the error term 
$Ef := E^0 f + E^1 f + E^2 f$ has operator norm less than $1$. Then $I + E$, as an operator on 
$\cH^{k, \A}_{-4+\gamma}(M \times [0,T_0],S)$, is invertible, and we set
\[ \mathcal{Q} := Q (I + E)^{-1}. \]
\end{proof}

%%%%%%%%%%%%%%%%%%%%%
\subsection{Extension of parametrices to the full time interval} \medskip
%%%%%%%%%%%%%%%%%%%%%

In this subsection we extend the existence results of Theorems \ref{Q neu} and \ref{R neu}
from the shorter time interval $[0,T_0]$ to the full time interval $[0,T]$. 
By Theorem \ref{Q neu}, for any $f\in \cH^{k, \A}_{-4 + \gamma} (M\times [0,T], S)$ there exists a $T_0 \in (0,T]$ and a 
solution $u = \mathcal{Q}f \in \cH^{k, \A}_{-2 + \gamma} (M\times [0,T_0], S)$ to the parabolic initial value problem
\[ (\partial_t + a^{-1}L) u = f, \quad u(t=0) = 0. \]
If $T_0 < T$, we consider the homogeneous Cauchy problem
\begin{equation*}
 (\partial_t + a^{-1}L) v_1 = 0, \quad v_1(t=0) = u(t=T_0), 
\end{equation*}
where the initial data $u(t=T_0) \in \cH^{k, \A}_{-2 + \gamma}(M,S)$.  
By Theorem \ref{R neu}, the solution to this problem, $v_1 = \mathcal{R} v_1(t=0)$, 
exists on the time interval $[0,T_0]$ independent of the initial value $u(t=T_0)$.  
We may use Theorem \ref{Q neu} to solve
\begin{equation*}
 (\partial_t + a^{-1}L) u_1 = f(p, t+T_0), u_1(0) = 0,
\end{equation*}
on the interval $[T_0, 2T_0]$, and then the function
\[ \widetilde{u}(p,t) = \begin{cases}
  u(p,t) & \text{for $0 \leq t \leq T_0$} \\
  u_1(p, t-T_0) + v_1(p,t-T_0) & \text{for $T_0 < t \leq 2T_0$}\\
\end{cases}
\]
extends $u$ past $T_0$.  This process continues until $nT_0 > T$, 
and produces a solution $u$ in $\cH^{k, \A}_{-2 + \gamma} (M\times [0,T], S)$.
Thus Theorem \ref{Q neu} and hence also Theorem \ref{R neu} hold on the
full time interval $[0,T]$.

\begin{rem}
To see that the solution produced in this way indeed has the claimed regularity (in particular including H\"older regularity in time), we observe that instead of piecing together solutions on the time intervals $[0,T_0], [T_0,2T_0]$ etc. we could have instead chosen overlapping intervals $[0,T_0], [T_0 - \ve, 2T_0 - \ve]$ etc. with a small $\ve > 0$. Then these solutions agree on the overlaps by a uniqueness argument analogous to the proof of Theorem \ref{main2-explizit}, based on the fact that the Friedrichs self adjoint extension $L$ is 
bounded from below.
\end{rem}

%%%%%%%%%%%%%%%%%%%%%
\subsection{Regularity of Ricci curvature along the flow} \medskip
%%%%%%%%%%%%%%%%%%%%%
We continue under the previously fixed notation where the 
upper script " $\widetilde{}$ " indicates that the quantity is taken with respect to the initial metric. 
The following result proves our main Theorem \ref{main2}.

\begin{thm}\label{main2-explizit}
Let $(M,g)$ be a tangentially stable conical manifold of dimension at least $4$, with $(\a,\g,k+1)$ H\"older regular geometry.
Let $g(t), t \in [0,T]$ be the solution of the Ricci de Turck flow with initial metric $g$ and reference metric
as in Theorem \ref{existence}. Then, assuming $\gamma \in (0,2+\min \, \{\mu_0,\mu_1\})$, if $\wt{\ric} \in \cH^{k, \A}_{-2 + \gamma} (M\times [0,T], S)$, then 
$\ric \in \cH^{k, \A}_{-2 + \gamma} (M\times [0,T], S)$. 	
\end{thm}	

\begin{proof}
By Proposition \ref{XR} we know
\begin{align*}
		(\p_t + a^{-1} L)\ric =: X(\ric) &\in 
		\mathcal{C}^{k-1,\a}_{\textup{ie}}(M \times [0,T],S)_{-4+2\ov{\g}} 
		\\ &\subseteq \cH^{k, \A}_{-4 + 2\ov{\gamma}} (M\times [0,T], S).
\end{align*}
Consider the parametrix $\mathcal{Q}$ of Theorem \ref{Q neu} and  
the parametrix $\mathcal{R}$ of Theorem \ref{R neu}. We set $\gamma' := \min \, \{ \gamma, 2\ov{\gamma}\}$ and define 
\begin{equation} \label{R chi defi}
\ric' := \mathcal{Q}(X(\ric)) + \mathcal{R}(\wt{\ric}) \in 
\cH^{k, \A}_{-2+\gamma'}(M \times [0,T],S),
\end{equation}
which is a solution of the parabolic initial value problem
\begin{equation} \label{evol equ param 1}
(\p_t + a^{-1}L) \ric' = X(\ric), \quad  \ric(t=0) = \wt{\ric}.
\end{equation}
We define  $u:= \ric' - \ric$. 
For $n\geq 3$ (recall that $\dim M \geq 4$) we can integrate by parts without 
boundary terms and conclude (recall $L=a^{1/2} \wt{\D} \, a^{-1/2}$)
$$
\partial_t \|u\|^2_{L^2(M,g_0)} =  - ( a^{-1/2} \wt{\D} \, a^{-1/2}u, u )_{L^2(M,g_0)}  
=  - \| \wt{\nabla}  (a^{-1/2} u) \|^2_{L^2(M,g_0)} \leq 0.
$$
Since $u(t=0)\equiv 0$, we find that $u\equiv 0$. Hence we still conclude, despite having no maximum principle at hand
$$ 
\ric \equiv \ric' \in \cH^{k, \A}_{-2+\gamma'}(M \times [0,T],S). 
$$
We iterate the argument, improving the weight as long as $\g' < \g$. 
This proves the theorem. 
\end{proof}

%%%%%%%%%%%%%%%%%%%%%%%%%%%
\section{Positivity of scalar curvature along the Ricci de Turck flow}\label{intro} \medskip
%%%%%%%%%%%%%%%%%%%%%%%%%%%

We can now prove our second main result on positivity of scalar curvature along the
singular Ricci de Turck flow. At this final step, we will need a stronger tangential 
stability hypothesis. We impose, cf. Remark \ref{DL1}, the following additional

\begin{assump}\label{strong-stability}
We assume strong tangential stability: $u_0, u_1 >n$, i.e. $\square_L > n$, so that 
we may choose $\g_0, \g_1 \geq 1$ satisfying \eqref{gamma01}. 
This stronger condition is studied in the Appendix \S \ref{strong-tangential},
where a list of examples is provided. Amongst the symmetric spaces of compact type, only
		\begin{align*}
		E_8,\qquad\mathrm{E}_7/[\mathrm{SU}(8)/\left\{\pm I\right\}],\qquad\mathrm{E}_8/\mathrm{SO}(16),\qquad\mathrm{E}_8/\mathrm{E}_7\cdot \mathrm{SU}(2)
		\end{align*}
satisfy the strong tangential stability condition\footnote{We hope to lift that restriction in the forthcoming work.}.
\end{assump}

The following theorem proves our second main result, Theorem \ref{main1}.

\begin{thm}\label{final-main-thm}
Let $(M,g)$ be a strong tangentially stable conical manifold of dimension at least $4$, with $(\a,\g,k+1)$ H\"older regular geometry,
where we assume $\gamma > 3$. Let $g(t), t \in [0,T]$ be the solution of the Ricci de Turck flow with initial metric $g$ and reference metric
as in Theorem \ref{existence}. Assume that $R_{g} \ge 0$. Then $R_{g(t)} \ge 0$ for all $t \in [0,T]$. Furthermore, if $R_{g}$ is positive at 
some point in the interior $M$, then $R_{g(t)}$ is positive in the interior $M$ for all $t \in (0,T]$.
\end{thm}

\begin{proof} 
By Theorem \ref{main2-explizit}, $\ric \in \cH^{k, \A}_{-2 + \gamma'} (M\times [0,T], S)$,
where $\gamma'$ is any weight smaller than $\{\gamma, 2+ \mu_0, 2+ \mu_1\}$. Due to 
Assumption \ref{strong-stability} of strong tangential stability, we find in 
 the evolution equation \eqref{evol equ} for the scalar curvature along the flow
\begin{equation}\label{higher order scalar}
\p_t R + \D R =  \la W, \n R \ra + 2 |\ric|^2 \in \mathcal{C}^\a_{\textup{ie}}(M\times [0,T]).
\end{equation}
We want to express the Laplace Beltrami operator $\D$ in terms of $a^{-1} \wt{\D}$,
where $\wt{\D}$ denotes the Laplace Beltrami operator with respect to $\wt{g}\equiv g(0)$. 
Consider first the Laplace Beltrami operator for $\wh{g} = a \wt{g}$,
\begin{align*}
\D_{\wh{g}} = - \frac{1}{\det \wh{g}} \, \p_i (\sqrt{\det \wh{g}} \cdot \wh{g}^{ij} \p_j) &= a^{-1} \wt{\D} - 
(\frac{m}{2} - 1) \cdot \p_i a \cdot a^{-2} \cdot \wt{g}^{ij} \p_j \\
&= a^{-1} \wt{\D} - \{x^{-1} \mathcal{V}_b a, a\} \, x^{-1} \mathcal{V}_b,
\end{align*}
where $\{x^{-1} \mathcal{V}_b a, a\} $ refers to a linear combination of monomials consisting of the terms in the brackets.
If we take into account higher order terms with $g= a \wt{g} \oplus \w$, we obtain in the same notation and higher regularity of 
$R$
\begin{equation} \label{higher order terms} \begin{split}
\D R - a^{-1} \wt{\D} R &= \{ x^{-1} \mathcal{V}_b a , a\} \, x^{-1} \mathcal{V}_b R + 
\{ x^{-1} \mathcal{V}_b \w , x^{-1}\w,\w\} \, x^{-1} \mathcal{V}_b R \\ &+ \{a, \w\} \, x^{-1} \mathcal{V}_b^2 \, R
\in \mathcal{C}^\a_{\textup{ie}}(M\times [0,T]).
\end{split}\end{equation}
The relation above is similar to Theorem \ref{L expansion}, where we 
note that the Lichnerowicz Laplacian acting on the pure trace component 
$S_1$ coincides with the Laplace Beltrami operator. Combining \eqref{higher order scalar}
and \eqref{higher order terms}, we conclude 
\begin{equation}\label{R-regularity-eqn}\begin{split}
\p_t  R + a^{-1} \wt{\D} R = P \in \mathcal{C}^\a_{\textup{ie}}(M\times [0,T]).
\end{split}\end{equation}
By \cite[Proposition 4.1, 4.6]{BahVer} there exists a solution $R' \in \mathcal{C}^\a_{\textup{ie}}(M\times [0,T])$ of \eqref{R-regularity-eqn}, 
for a given initial value $R(0) \in \mathcal{C}^\a_{\textup{ie}}(M)$, such that $\D R' \in \mathcal{C}^\a_{\textup{ie}}(M\times [0,T])$ again.
Exactly as in Theorem \ref{main2-explizit}, we find $R=R'$. Hence the maximum principle obtained in \cite[Theorem 3.1]{BahVer} 
applies to $R$ and, denoting by $R_{\min}(t)$ the minimum of the scalar curvature at time $t$, which is attained at
$p_{\min}(t) \in \ov{M}$, we conclude
\begin{align*}
\p_t R_{\min}(t) &\geq  \la W(p_{\min}(t)), \n R_{\min}(t) \ra + 2 |\ric(p_{\min}(t))|^2 \\
&\geq  \la W(p_{\min}(t)), \n R_{\min}(t) \ra 
\end{align*}
Now, if $p_{\min}(t)$ lies in the open interior of $M$, then $\n R_{\min}(t) = 0$.
If $p_{\min}(t)$ lies at the conical singularity, we argue as follows: note that $\la W, \n R \ra \in \mathcal{C}^\a_{\textup{ie}}(M\times [0,T])_{\gamma'-3}$.
Since by strong tangential stability, we can choose $\gamma' > 3$, we find that $\la W, \n R \ra$ is vanishing at the conical singularity and hence 
\begin{align*}
\p_t R_{\min}(t) \geq  \la W(p_{\min}(t)), \n R_{\min}(t) \ra = 0.
\end{align*}
This implies that $R \ge 0$ for all $t \in [0,T]$. 
Now the last statement follows from the strong maximum principle \cite[Theorem A.5]{kleilo}. 
\end{proof}

%%%%%%%%%%%%%%%%%
\section*{Appendix: Characterizing the spectral conditions}\label{strong-tangential}
%%%%%%%%%%%%%%%%%

In this appendix, we aim to characterize strong tangential stability
% eigenvalue condition $\square_L>n$ 
in 
terms of eigenvalues of geometric operators on the cross-section of a cone. Note that
this condition is only used in the last part of our argument, proving Theorem \ref{main2}.
It is not used in our first main result Theorem \ref{main1} on higher regularity of Ricci curvature.
\medskip
	
Note that the operator $\square_L$ can be entirely be described as an operator on the cross 
section $(F^n,g_F)$ of the cone. Therefore in this section, all scalar products and geometric 
operators are taken with respect to $(F^n,g_F)$. The operator 
 $\Delta_E$ denotes the Einstein operator on symmetric two-tensors over $F$.
It is given by \[ 
\D_E\w_{ij} = \D\w_{ij} - 2 g^{pq} \Rm^r_{qij} \w_{rp} ,
\]
where $\Delta$ is 
the rough Laplacian. We write $\Delta$ for the Laplace Beltrami operator on $F$. 
Moreover, $TT$ denotes the space of symmetric two-tensors which are trace-free and 
divergence-free at each point.
	\begin{thm}\label{tang_op_estimate}
		Let $(F^n,g_F)$, $n\geq 3$ be a compact Einstein manifold with constant $n-1$. 
		Then $(F,g_F)$ is strongly tangentially stable if and only if we have
		the conditions
		\begin{align*}
		\mathrm{Spec}(\Delta_E|_{TT})&>n,\\
		\mathrm{Spec}(\Delta_{\Omega^1(F)\cap\ker(div)})&>n+\sqrt{n^2+2n+1},
		\end{align*}
		and if for all positive eigenvalues $\lambda$ of the Laplace-Beltrami operator on functions satisfies
		\begin{align*}
		& \quad n(\lambda-3n+2)(\lambda+4-n)n(\lambda+n+2)\\& \quad -8n(n-1)(\lambda-n)(\lambda+n+2)-8\lambda n(n+1)(\lambda-3n+2)>0.
		\end{align*}
	\end{thm}
	\begin{proof}
		We first recall from the discussion before Remark \ref{DL1} that strong tangential stability is equvalent to the two estimates 
		\begin{align*}
	u_0=\min(\Spec \, \square'_L \backslash\{0\})>n,\qquad u_1= \min(\Spec \, \D_F \backslash\{0\}) >n.
		\end{align*}
		However, the condition $u_1>n$ holds for any Einstein metric except the sphere, where we have equality \cite{Obata}.		
For the rest of the proof, it thus suffices to consider the bundle $S_0 \restriction F$. 
We use the same methodology as in \cite{Klaus-Vertman} 
which builds up on a decomposition of symmetric $2$-tensors established in \cite{Klaus}. 
We use the notation in \cite[Section 2]{Klaus} and and the calculations in Section 3 of the 
same paper where  we  remove  all  terms  containing  radial derivatives  in order  to  obtain expressions for the tangential operator. 
More precisely, we write
		\begin{equation}
		\begin{split}
		\{h_i\} \ &\textup{- basis of} \ L^2(TT), \quad \Delta_E h_i = \kappa_i h_i, \quad V_{1,i} := \langle r^2 h_i \rangle, \\
		\{\omega_i\} \ &\textup{- basis of coclosed sections} \ L^2(T^*F), \quad \Delta \omega_i = \mu_i \omega_i, \\
		&\quad V_{3,i} :=  \langle r^2 \delta^* \omega_i\rangle \oplus \langle dr \odot r \omega_i\rangle, \\
		\{v_i\} \ &\textup{- basis of} \ L^2(F), \quad \Delta v_i = \lambda_i v_i, \\
		&\quad V_{4,i} := \langle r^2 (n \nabla^2 v_i + \Delta v_i g)\rangle \oplus \langle dr \odot r \nabla v_i\rangle
		\\ & \quad \oplus \langle v_i (r^2 g - ndr \otimes dr)\rangle.
		\end{split}
		\end{equation}
		Here, $\Delta$ in $\Delta \omega_i$ denotes the connection Laplacian, while $\Delta$ in $\Delta v_i$
		denotes the Laplace Beltrami operator.
		The spaces $V_{1,i}, V_{3,i}, V_{4,i}$, with $L^2(0,1)$ coefficients, span all trace-free sections $L^2(S_0 \restriction F)$ over $F$,
		and are invariant under the action of the Lichnerowicz Laplacian. In \cite[Section 2]{Klaus}, there is also a notion for the spaces $V_{2,i}:=\langle v_i (r^2 g + ndr \otimes dr)\rangle$. But these spaces span the full trace secions $L^2(S_1 \restriction F)$, whose discussion is not relevant here.
		At first, if $\wt{h} = r^2 h_i \in V_{1,i}$,
		\begin{align*}
		(\square_L\wt{h},\wt{h})_{L^2}=\kappa_i\Vert \wt{h} \Vert_{L^2}
		\end{align*}
		such that $\square_L>n$ on $V_{1,i}$ for all $i$ if and only if $\kappa_i>n$ 
		for all eigenvalues of the Einstein operator on $TT$-tensors are positive (non-negative).\medskip
		
		Let now $\wt{h} \in V_{3,i}$ so that it is of the form
		$\wt{h} = \wt{h}_1 + \wt{h}_2 = \varphi r^2 \delta^* \omega_i + \psi dr \odot r \omega_i$. In this case, we have the scalar products
		\begin{align*}	(\square_L \wt{h}_1, \wt{h}_1)_{L^2} &= \frac{\varphi^2}{2} (\mu_i - (n-1))^2,\\
		(\square_L \wt{h}_2, \wt{h}_2)_{L^2}& = \psi^2 [2\mu_i + (2n+6)],\\
		(\square_L \wt{h}_1, \wt{h}_2)_{L^2}& = -2 (\mu_i - (n-1)) \psi \varphi.
		\end{align*}
		Taking $r^2 \delta^* \omega_i$ and $dr \odot r \omega_i$ as a basis, $\square_L$ respects the subspace and acts as $2 \times 2$-matrix
		\begin{equation*}
		\left( \begin{array}{cc}
		\frac{1}{2} (\mu_i - (n-1))^2 & -2(\mu_i - (n-1))\\
		-2(\mu_i - (n-1)) & 2 \mu_i + (2n + 6)
		\end{array} \right).
		\end{equation*}
		Because
		\begin{equation*}
		\begin{split}
		\Vert \wt{h}_1 \Vert_{L^2}^2 & = \frac{1}{2} (\mu_i - (n-1)) \cdot \vert \varphi \vert^2,\qquad
		\Vert \wt{h}_2 \Vert_{L^2}^2  = 2 \vert \psi \vert^2,
		\end{split}
		\end{equation*}
		the operator $\square_L-n\cdot {id}>0$ is represented by
		\begin{equation*}
		A:=\left( \begin{array}{cc}
		\frac{1}{2} (\mu_i - (n-1))^2 -\frac{n}{2} (\mu_i - (n-1)) & -2(\mu_i - (n-1))\\
		-2(\mu_i - (n-1)) & 2 \mu_i + 6
		\end{array} \right)	.
		\end{equation*}
		The matrix $A$ is positive definite if and only if the matrix
		\begin{equation*}
		B:=\left( \begin{array}{cc}
		\frac{1}{2} (\mu_i - (2n-1)) & -2(\mu_i - (n-1))\\
		-2 & 2 \mu_i + 6
		\end{array} \right)
		\end{equation*}
		is positive definite
		because $A$ is obtained from $B$ by multiplying the first column by $\mu_i - (n-1)$.
		By computing principal minors, this holds if 
		\begin{align*}
		\det(B)=\frac{1}{2} (\mu_i - (2n-1)) (2 \mu_i + 6)-4(\mu_i - (n-1))>0.
		\end{align*}
		This in turn holds if
		\begin{align*}
		\mu_i>n+\sqrt{n^2+2n+1}.
		\end{align*}
		Therefore, $\square_L>n$  on the spaces $ V_{3,i}$ if and only if $\Delta>n+\sqrt{n^2+2n+1}$ on coclosed sections $L^2(T^*F)$.
		It remains to consider the case $\wt{h}\in V_{4,i}$, so that it is of the form
		\begin{align*}\wt{h} = \wt{h}_1 + \wt{h}_2 + \wt{h}_3 = \varphi r^2 (n \nabla^2 v_i + \Delta v_i g) + \psi dr \odot r \nabla v_i + \mathcal{X} v_i (r^2 g - ndr \otimes dr).
		\end{align*}
		This case is the most delicate one. We have the scalar products	\begin{equation*}
		\begin{split}
		(\square_L \wt{h}_1,\wt{h}_1) & = n(n-1) \lambda_i (\lambda_i-n) (\lambda_i - 2(n-1)) \varphi^2,\\
		(\square_L \wt{h}_2,\wt{h}_2) & = [2\lambda_i (\lambda_i - (n-1)) + (2n+6) \lambda_i] \psi^2,\\
		(\square_L \wt{h}_3,\wt{h}_3) & = [n\{(n+1) \lambda_i - 2(n+1) \} + 2n^2(n+3)] \mathcal{X}^2,\\
		(\square_L \wt{h}_1,\wt{h}_2) & = -4(n-1) \lambda_i (\lambda_i -n) \psi \varphi,\\
		(\square_L \wt{h}_2,\wt{h}_3) & = 4(n+1) \lambda_i \psi \mathcal{X},\\
		(\square_L \wt{h}_1,\wt{h}_3) & = 0
		\end{split}
		\end{equation*}
		and the norms
		\begin{equation*}
		\begin{split}
		\Vert \wt{h}_1 \Vert_{L^2}^2 & = n(n-1) \lambda_i (\lambda_i - n) \varphi^2,\\
		\Vert \wt{h}_2 \Vert_{L^2}^2 & = 2 \varphi^2 \lambda_i,\\
		\Vert \wt{h}_3 \Vert_{L^2}^2 & = (n+1)n.
		\end{split}
		\end{equation*}
		Consider $(\square_L - n\cdot id)$. It acts as a matrix $A=(a_{ij})_{1\leq n\leq 3}$,
		whose coefficients are given by
		\begin{align*}
		a_{11}&=n(n-1)\lambda_i (\lambda_i - n)[\lambda_i -2(n-1)-n],\\
		a_{22}&=2 \lambda_i [\lambda_i - (n-1)  + 3],\\
		a_{33}&= n\{(n+1) \lambda_i - 2(n-1) - n(n+1) + 2n(n+3) \},\\
		a_{12}&=a_{21}=-4(n-1)\lambda_i (\lambda_i -n),\\
		a_{23}&=a_{32}=4(n+1)\lambda_i,\\
		a_{13}&=a_{31}=0.
		\end{align*}
		In order to prove positivity of this matrix, we consider its principal minors 
		$A_{33}$ (which is the lower right entry), $A_{23}$ (the lower right $2\times 2$-matrix) and $A$ (the whole matrix).
		At first,
		\begin{equation*}
		\begin{split}
		A_{33} & = n \{(n+1) \lambda_i + n^2 +3n+2 \} > 0.
		\end{split}
		\end{equation*}
		Observe that in the case $\lambda_i=0$, $\wt{h}_1 \equiv 0$ and $ \wt{h}_2 \equiv 0$, so that $V_{4i} = span\{\wt{h}_3\}$ and
		hence, $(\square_L-n\cdot id)$ acts as $A_{33} > 0$. Therefore, we may from now on assume 
		that $\lambda_i > 0$, which means that actually $\lambda_i \ge n$ (due to eigenvalue estimates 
		for Einstein manifolds, see e.g. \cite{Obata}) with $\lambda_i = n$ only for $\mathbb{S}^n$. By considering the matrix
		\begin{equation*}
		\left( \begin{array}{cc}
		2(\lambda_i + 4 - n) & 4(n+1)\lambda_i\\
		4(n+1) & n\{(n+1) \lambda_i  + n^2 +3n+2\}
		\end{array} \right),
		\end{equation*}
		from which one recovers $A_{23}$ by multiplying the first column by $\lambda_i$,
		we see that
		\begin{equation*}
		\begin{split}
		\frac{\det A_{23}}{\lambda_i} & = 2(\lambda_i + 4 - \epsilon)n \cdot [(n+1)\lambda_i + n^2 +3n+2] - 16(n+1)^2 \lambda_i\\
		& = 2n(n+1)\lambda_i^2-4(n+1)(n+4)\lambda_i-2n(n-4)(n^2+3n+2),
		\end{split}
		\end{equation*}
		which is positive if
		\begin{align*}
		\lambda_i>\frac{n+4}{n}+\sqrt{\frac{n+4}{n}+(n-4)(n+2)}.
		\end{align*}
		Here, the right hand side is smaller than $n$ if $n\geq4$ such that this condition holds anyway. 
		Before we compute the full determinant of $A$, we remark that in the case $\lambda_i=n$, the tensor $\wt{h}_1$ 
		is vanishing so that in this case, the matrix $A$ describing $\square_L$ on $V_{4,i}$ reduces to the matrix 
		$A_{23}$ which just has been considered. Therefore, there is nothing more to prove in this case and we may 
		assume $\lambda=\lambda_i>n$ from now on.
		To compute the full determinant of $A$, we first consider the matrix
		\begin{align*}
		\left( \begin{array}{ccc}
		n[\lambda - 2(n-1) - n] & -2(n-1)(\lambda - n) & 0\\
		-4 & \lambda-(n-1)+3 & 4 \lambda\\
		0& 2(n+1) & n\{\lambda + 2(n+1) -n  \}
		\end{array} \right)
		\end{align*}
		from which we recover $A$ by mutiplying the three columns by $(n-1)\lambda(\lambda-n)$ and $2\lambda,(n+1)$, respectively.
		We get
		\begin{equation*}
		\begin{split}
		[(n-1) \lambda (\lambda -n) 2 \cdot &\lambda \cdot (n+1)]^{-1} \det A
		= n(\lambda-3n+2)(\lambda+4-n)n(\lambda+n+2)\\ &-8n(n-1)(\lambda-n)(\lambda+n+2)-8\lambda n(n+1)(\lambda-3n+2)
		\end{split}
		\end{equation*}
		which finishes the proof.
	\end{proof}
	\begin{thm}
		Amongst the symmetric spaces of compact type, only
		\begin{align*}
		E_8,\qquad\mathrm{E}_7/[\mathrm{SU}(8)/\left\{\pm I\right\}],\qquad\mathrm{E}_8/\mathrm{SO}(16),\qquad\mathrm{E}_8/\mathrm{E}_7\cdot \mathrm{SU}(2)
		\end{align*}
		are strongly tangentially stable.\end{thm}
	\begin{proof}
		We merge and analyse the Tables $2$ and $3$ in \cite{Klaus} and the Tables $1$ and $2$ in \cite{CaoHe}.
		In the tables below, $\Lambda$ denotes the smallest nonzero eigenvalue of the Laplace-Beltrami operator divided by the Einstein constant $n-1$ and $\Theta$ denotes the smallest eigenvalue of the Lichnerowicz Laplacian $\Delta_L$ on symmetric $2$-tensors with $\int_F tr h\text{ }dV=0$ divided by the same  constant $n-1$. \medskip
		
		To check that the above mentioned spaces satisfy the condition of Theorem \ref{tang_op_estimate}, we check on one hand that the estimate for $\lambda$ in this theorem holds for $(n-1)\cdot \Lambda$. On the other hand, we have the relations $\Delta_L=\Delta_E+2(n-1)\mathrm{id}$, $\delta^*\circ \Delta_H=\Delta_L\circ \delta^*$ and $\Delta_H=\Delta+(n-1)\mathrm{id}$. Therefore, the condition $\mathrm{Spec}(\Delta_E|_{TT})>n$ holds, if $(n-1)\Theta-2(n-1)>n$ and the condition $\mathrm{Spec}(\Delta_{\Omega^1(F)\cap\ker(div)})>n+\sqrt{n^2+2n+1}$ holds if $(n-1)\Theta>2n-1+\sqrt{n^2+2n+1}$. \medskip
		
		To show that all the other examples do not satisfy $\square_L>n$, we proceed as follows: We check that $(n-1)\Lambda$ satisfies  \begin{align*}
		\Lambda\leq \frac{n+4}{n}+\sqrt{\frac{n+4}{n}+(n-4)(n+2)}
		\end{align*}
		or
		\begin{align*}
		&n(\Lambda-3n+2)(\Lambda+4-n)n(\Lambda+n+2)\\&-8n(n-1)(\Lambda-n)(\Lambda+n+2)-8\Lambda n(n+1)(\Lambda-3n+2)\leq0.
		\end{align*}
		This holds for example, if $\Lambda\leq 3$. Because $\Delta_L(f\cdot g_F)=\Delta f\cdot g_F$ and $\Delta_L\circ \nabla^2=\nabla^2\circ \Delta$, we clearly have $\Theta\leq\Lambda$.
Because of the decomposition
\begin{align*}
C^{\infty}(\textup{Sym}^2(T^*F))=C^{\infty}(F)\cdot g_F\oplus \nabla^2(C^{\infty}(F))\oplus \delta^*(\Omega^1(F)\cap \ker(div))\oplus TT,
\end{align*}
(c.f. \cite{Klaus})
$\Theta$ is attained on
\begin{align*}
\delta^*(\Omega^1(F)\cap \ker(div))\oplus TT
\end{align*}
if $\Theta<\Lambda$. In this case, we check if $\Theta$ satisfies $(n-1)\Theta-2(n-1)\leq n$ and $(n-1)\Theta<2n-1+\sqrt{n^2+2n+1}$ (both conditions are satisfied if $\Theta\leq 3$). Due to the relations above, this implies that either $\mathrm{Spec}(\Delta_E|_{TT})>n$ or $\mathrm{Spec}(\Delta_{\Omega^1(F)\cap\ker(div)})>n+\sqrt{n^2+2n+1}$ fails so that $\square_L>n$ fails. \medskip

Therefore, the condition $\mathrm{Spec}(\Delta_E|_{TT})>n$ holds, if $(n-1)\Theta-2(n-1)>n$ and the condition $\mathrm{Spec}(\Delta_{\Omega^1(F)\cap\ker(div)})>n+\sqrt{n^2+2n+1}$ holds if $(n-1)\Theta>2n-1+\sqrt{n^2+2n+1}$. \medskip

To finish the proof, one just has to go through all the values of $\Theta$ and $\Lambda$ in the table below. 
Strong tangential stability is abbreviated by STS.
As was already said, for the spaces mentioned in the statement of the theorem, one manually checks the estimates above to verify that the conditions of Theorem \ref{tang_op_estimate} are satisfied. In all the other cases, one finds  $\Theta\leq3$ except in the case  $\mathrm{E}_7/\mathrm{SO}(12)\cdot\mathrm{SU}(2)$ where one manually checks the condition on $\Lambda$ mentioned above. This finishes the proof of the theorem.
		
		\begin{center}
			\renewcommand{\arraystretch}{1.5}
			\begin{longtable}{|l|l|c|c|c|l|}
				\hline
				type & $\mathrm{G}$ & $\dimn(\mathrm{G})$ & $\Lambda$ & $\Theta$ & STS \\
				\hline
				$\mathrm{A}_p$	& $\mathrm{SU}(p+1)$, $p\geq 2$			& $p^2-1$			
				& $\frac{2p(p+2)}{(p+1)^2}$ & $\frac{2p(p+2)}{(p+1)^2}$  & no\\
				\hline
				\multirow{3}{*}{$\mathrm{B}_n$}
				& $\mathrm{Spin}(5)	$		& $10$		& $\frac{5}{3}$ & $\frac{4}{3}$ & no\\
				& $\mathrm{Spin}(7)	$		& $21$	& $\frac{21}{10}$ & $\frac{12}{5}$& no \\
				& $\mathrm{Spin}(2p+1)$, $n\geq 4$		& $2p(p+1)$		& $\frac{4p}{2p-1}$ & $\frac{4p}{2p-1}$ & no\\
				\hline
				$\mathrm{C}_p$	& $\mathrm{Sp}(p)$, $p\geq 3$			& $p(2p+1)$		& $\frac{2p+1}{p+1}$ & $\frac{4p-1}{2(p+1)}$ & no\\
				\hline
				$\mathrm{D}_p$	& $\mathrm{Spin}(2p)$, $p\geq 3$			& $p(2p+1)$		& $\frac{2p-1}{p-1}$ & $\frac{2p-1}{p-1}$ & no\\
				\hline
				$\mathrm{E}_6$	& $\mathrm{E}_6$			& $156$		& $\frac{26}{9}$ & $\frac{17}{6}$ & no\\
				\hline
				$\mathrm{E}_7$	& $\mathrm{E}_7$			& $266$		& $\frac{19}{6}$ & $3$ & no\\
				\hline
				$\mathrm{E}_8$	& $\mathrm{E}_8$			& $496$		& $4$ & $\frac{47}{15}$ & yes\\
				\hline
				$\mathrm{F}_4$	& $\mathrm{F}_4$			& $52$		& $\frac{8}{3}$ & $\frac{8}{3}$ & no\\
				\hline
				$\mathrm{G}_2$	& $\mathrm{G}_2$			& $14$		& $2$ & $2$ & no\\
				\hline	
				\caption{Conditions of Theorem \ref{tang_op_estimate} for simple Lie groups}
			\end{longtable}
		\end{center}
		In the case of irreducible rank-$1$ symmetric spaces of compact type, an analogous argumentation yields the following table:
		\begin{center}
			\renewcommand{\arraystretch}{1.5}
			\begin{longtable}{|l|l|c|c|c|l|}
				\hline
				type & $G/K$ & $\dimn(G/K)$ & $\Lambda$ & $\Theta$ & STS \\
				\hline	
				\multirow{2}{*}{A I}	& $\mathrm{SU}(p)/\mathrm{SO}(p)$, $5\geq p\geq 3$			
				& $\frac{(p-1)(p+2)}{2}$			& $\frac{2(p-1)(p+2)}{p^2}$ &  $2$ &  no\\
				& $\mathrm{SU}(p)/\mathrm{SO}(p)$, $p\geq 6$			
				& $\frac{(p-1)(p+2)}{2}$			& $\frac{2(p-1)(p+2)}{p^2}$ &  $2$ &  no\\
				\hline
				\multirow{2}{*}{A II}
				& $\mathrm{SU}(4)/\mathrm{Sp}(2)=S^5	$		& $5$		& $\frac{5}{4}$ & $3$ & no \\
				& $\mathrm{SU}(2p)/\mathrm{Sp}(p)$, $p\geq3	$		& $2p^2-p-1$	& $\frac{(2p+1)(p-1)}{p^2}$ & $2$ & no \\
				\hline
				\multirow{2}{*}{A III}
				& $\frac{\mathrm{U}(p+1)}{\mathrm{U}(p)\times \mathrm{U}(1)}=\CP^p	$		& $2p$		& $2$ & $2$ & no\\
				& $\frac{\mathrm{U}(p+q)}{\mathrm{U}(q)\times \mathrm{U}(p)}$, $q\geq p\geq2	$		
				& $2pq$	& $2$ &  $2$ & no \\
%				\hline
%				\pagebreak
							\hline
				\multirow{6}{*}{B I}
				& $\frac{\mathrm{SO}(5)}{\mathrm{SO}(3)\times SO(2)}$		& $6$	& $2$ &   $\frac{4}{3}$ & no \\		
				& $\frac{\mathrm{SO}(2p+3)}{\mathrm{SO}(2p+1)\times \mathrm{SO}(2)}$, $p\geq 2$		& $4p+2$	& $2$ &   $\frac{8}{2p+1}$ & no \\
				& $\frac{\mathrm{SO}(7)}{\mathrm{SO}(4)\times \mathrm{SO}(3)}$		& $12$	& $\frac{12}{5}$ &  $\frac{8}{5}$ & no \\
				& $\frac{\mathrm{SO}(2p+3)}{\mathrm{SO}(3)\times \mathrm{SO}(2p)}$, $p\geq 3$		
				& $6p$	& $\frac{4p+6}{2p+1}$ &  $\frac{8}{2q+1}$ & no \\
				& $\frac{\mathrm{SO}(2q+2p+1)}{\mathrm{SO}(2q+1)\times \mathrm{SO}(2p)}$, $p,q\geq 2$		
				& $2n(2m+1)$	& $\frac{4m+4n+2}{2m+2n-1}$ & $\frac{8}{2p+2q-1}$ & no \\
				\hline
				\multirow{1}{*}{B II} & $\frac{\mathrm{SO}(2p+1)}{\mathrm{SO}(2p)}=\mathbb{S}^{2p}$, $p\geq 1$		
				& $2p$		& $\frac{2p}{2p-1}$ & $\frac{4p+2}{2p-1}$ & no \\	
				\hline
				C I	& $\mathrm{Sp}(p)/\mathrm{U}(p)$, $p\geq 3$		& $p(p+1)$		& $2$ & $\frac{2p}{p+1}$ & no\\		
				\hline
				\multirow{3}{*}{C II}
				& $\frac{\mathrm{Sp}(2)}{\mathrm{Sp}(1)\times \mathrm{Sp}(1)}=\mathbb{S}^4	$		
				& $4$		& $\frac{4}{3}$ & $\frac{10}{3}$ & no \\
				& $\frac{\mathrm{Sp}(p+1)}{\mathrm{Sp}(p)\times \mathrm{Sp}(1)}=\HP^p$, $p\geq 2$		
				& $4p$	& $\frac{2(p+1)}{p+2}$ & $\frac{2(p+1)}{p+2}$ & no \\
				& $\frac{\mathrm{Sp}(p+q)}{\mathrm{Sp}(q)\times \mathrm{Sp}(p)}$, $q\geq p\geq 2$		
				& $4pq$	& $\frac{2(p+q)}{p+q+1}$ & $\frac{2(p+q)}{p+q+1}$ & no \\
				\hline
								\pagebreak
									\hline
								type & $G/K$ & $\dimn(G/K)$ & $\Lambda$ & $\Theta$ & STS \\
				\hline
				
				\multirow{6}{*}{D I}
				
				& $\frac{\mathrm{SO}(8)}{\mathrm{SO}(5)\times\mathrm{SO}(3)}$		& $15$	& $\frac{5}{2}$ &  $\frac{5}{2}$ & no \\		
				& $\frac{\mathrm{SO}(2p+2)}{\mathrm{SO}(2p)\times \mathrm{SO}(2)}$, $p\geq 3$		& $4p$	& $2$ & $2$ & no \\
				& $\frac{\mathrm{SO}(2p)}{\mathrm{SO}(p)\times \mathrm{SO}(p)}$, $p\geq 4$		& $p^2$	& $\frac{2p}{p-1}$ & $\frac{2p}{p-1}$ & no \\
				& $\frac{\mathrm{SO}(2p+2)}{\mathrm{SO}(p+2)\times \mathrm{SO}(p)}$, $p\geq 4$		& $p(p+2)$	& $\frac{2p+2}{p}$ & $\frac{2p+2}{p}$ & no \\
				& $\frac{\mathrm{SO}(2p)}{\mathrm{SO}(2p-q)\times \mathrm{SO}(q)}$,	& $(2p-q)q$	& $\frac{2p}{p-1}$ & $\frac{2p}{p-1}$ & no \\
				& $p-2\geq q\geq 3$	& & & & \\
				\hline
				D II& $\frac{SO(2p+2)}{SO(2p+1)}=S^{2p+1}$, $p\geq 3$		& $2p+1$		& $\frac{2p+1}{2p}$ & $\frac{2(p+1)}{p}$ & no \\
				\hline
				D III	& $\mathrm{SO}(2p)/\mathrm{U}(p)$, $p\geq 5$		& $p(p-1)$		& $2$ &  $2$ & no\\
				\hline
				E I	& $\mathrm{E}_6/[\mathrm{Sp}(4)/\left\{\pm I\right\}]$ & $42$		& $\frac{28}{9}$ & $3$ & no\\
				\hline
				E II	& $\mathrm{E}_6/\mathrm{SU}(2)\cdot \mathrm{SU}(6)$			& $40$		& $3$ & $3$ & no\\
				\hline
				E III	& $\mathrm{E}_6/\mathrm{SO}(10)\cdot \mathrm{SO}(2)$	& $32$		& $2$ &  $2$ & no\\
				\hline
				E IV	& $\mathrm{E}_6/\mathrm{F}_4$			& $26$		& $\frac{13}{9}$ & $\frac{13}{9}$ & no\\
				\hline
				E V	& $\mathrm{E}_7/[\mathrm{SU}(8)/\left\{\pm I\right\}]$ & $70$		& $\frac{10}{3}$ & $\frac{28}{9}$ & yes \\
				\hline
				E VI	& $\mathrm{E}_7/\mathrm{SO}(12)\cdot\mathrm{SU}(2)$ & $64$		& $\frac{28}{9}$ & $\frac{28}{9}$ & no \\
%				\hline
%				\pagebreak
%				\hline
%				type & $G/K$ & $\dimn(G/K)$ &  $\Lambda$ & $\Theta$ & STS \\
				\hline
				E VII	& $\mathrm{E}_7/\mathrm{E}_6\cdot \mathrm{SO}(2)$ & $54$		& $2$ & $2$ & no \\
				\hline
				E VIII	& $\mathrm{E}_8/\mathrm{SO}(16)$ & $128$		& $\frac{62}{15}$ & $\frac{16}{5}$ & yes \\
				\hline
				E IX	& $\mathrm{E}_8/\mathrm{E}_7\cdot \mathrm{SU}(2)$ & $112$		& $\frac{16}{5}$ & $\frac{16}{5}$ & yes \\
				\hline
				F I	& $\mathrm{F}_4/Sp(3)\cdot\mathrm{SU}(2)$ & $28$		& $\frac{26}{9}$ & $\frac{26}{9}$ & no \\
				\hline
				F II	& $\mathrm{F}_4/\mathrm{Spin}(9)$ & $16$		& $\frac{4}{3}$ & $\frac{4}{3}$ & no \\
				\hline
				G	& $\mathrm{G}_2/\mathrm{SO}(4)$ & $8$		& $\frac{7}{3}$ & $\frac{7}{3}$ & no \\
				\hline
				\caption{Conditions of Theorem \ref{tang_op_estimate} for symmetric spaces of non-group type}
			\end{longtable}
		\end{center}
	\end{proof}

	\providecommand{\url}[1]{\texttt{#1}}
	\expandafter\ifx\csname urlstyle\endcsname\relax
	\providecommand{\doi}[1]{doi: #1}\else
	\providecommand{\doi}{doi: \begingroup \urlstyle{rm}\Url}\fi

\end{document}